\documentclass[reqno, 11pt]{amsart}

\usepackage[parfill]{parskip}    
\usepackage{amsrefs, amsmath}
\usepackage{amsfonts, mathrsfs}
\usepackage{amssymb}
\usepackage{amscd, mathtools}
\usepackage{hyperref}
\usepackage{cleveref, multirow} 
\usepackage{fullpage, color}
\usepackage[labelformat=empty]{subfig}
\usepackage{booktabs}
\usepackage{xypic}
\usepackage{enumerate, longtable}
\usepackage{makecell, diagbox}
\usepackage{IEEEtrantools}
\usepackage{pifont}

\usepackage{enumitem}

\numberwithin{equation}{section}

\BibSpec{article}{%
+{}{\sc \PrintAuthors} {author}
+{,}{ \textrm} {title}
+{,}{ \textit} {journal}
+{,}{ } {volume}
+{}{ \IfEmptyBibField{journal}{}{\PrintDate{date}}} {transition}
+{,}{ } {pages}
+{,}{ } {status}
+{.}{} {transition}
+{}{ Preprint available at \tt} {eprint}
+{.}{} {transition}
}

\BibSpec{book}{%
+{}{\sc \PrintAuthors} {author}
+{,}{ \textit} {title}
+{.}{ } {series}
+{,}{ } {volume}
+{.}{ } {publisher}
+{,}{ } {place}
+{,}{ } {date}
+{.}{} {transition}
}

\BibSpec{collection.article}{
+{} {\sc \PrintAuthors} {author}
+{,} { \textrm } {title}
+{,} { in \it \PrintConference} {conference}
+{,} { pp.~} {pages}
+{.} {\PrintBook} {book}
+{,} { \PrintDateB} {date}
+{.} {} {transition}
}

\BibSpec{innerbook}{
+{.} { \emph} {title}
+{.} { } {part}
+{:} { \emph} {subtitle}
+{.} { } {series}
+{,} { } {volume}
+{.} { Edited by \PrintNameList} {editor}
+{.} { Translated by \PrintNameList}{translator}
+{.} { \PrintContributions} {contribution}
+{.} { } {publisher}
+{.} { } {organization}
+{,} { } {address}
+{,} { \PrintEdition} {edition}
+{,} { \PrintDateB} {date}
+{.} { } {note}
}

\allowdisplaybreaks[3]

\crefname{equation}{Eq.}{Eqs.}
\crefname{eqnarray}{Eq.}{Eqs.}
\crefname{algo}{Algorithm}{Algorithms}
\crefname{conj}{Conjecture}{Conjectures}
\crefname{lem}{Lemma}{Lemmas}
\crefname{thm}{Theorem}{Theorems}
\crefname{claim}{Claim}{Claims}
\crefname{rmk}{Remark}{Remarks}
\crefname{prop}{Proposition}{Propositions}
\crefname{section}{Section}{Sections}
\crefname{appendix}{Appendix}{Appendices}
\crefname{cor}{Corollary}{Corollaries}
\crefname{figure}{Figure}{Figures}
\crefname{table}{Table}{Tables}
\crefname{example}{Example}{Examples}
\crefname{prob}{Problem}{Problems}
\crefname{assm}{Assumption}{Assumptions}
\crefname{defn}{Definition}{Definitions}

\newcommand{\ri}{\mathrm{i}}
\newcommand{\re}{\mathrm{e}}

\newcommand{\bbA}{\mathbb{A}}

\newcommand{\bbN}{\mathbb{N}}

\newcommand{\bbL}{\mathbb{L}}
\newcommand{\bbZ}{\mathbb{Z}}
\newcommand{\bbR}{\mathbb{R}}
\newcommand{\bbC}{\mathbb{C}}
\newcommand{\bbP}{\mathbb{P}}

\newcommand{\bbQ}{\mathbb{Q}}

\def\bary{\begin{array}} 
\def\eary{\end{array}} 
\def\ben{\begin{enumerate}} 
\def\een{\end{enumerate}}
\def\bit{\begin{itemize}} 
\def\eit{\end{itemize}}
\def\nn{\nonumber} 

\newcommand{\cY}{\mathcal{Y}}
\newcommand{\cZ}{\mathcal{Z}}
\newcommand{\cO}{\mathcal{O}}

\newcommand{\cP}{\mathcal{P}}

\newcommand{\cW}{\mathcal{W}}

\def\beq{\begin{equation}}                     %
\def\eeq{\end{equation}}                       %
\def\bea{\begin{eqnarray}}                     
\def\eea{\end{eqnarray}}
\def\bary{\begin{array}} 
\def\eary{\end{array}} 
\def\ben{\begin{enumerate}} 
\def\een{\end{enumerate}}
\def\bit{\begin{itemize}} 
\def\eit{\end{itemize}}
\def\nn{\nonumber}

\def\a{\alpha}
\def\b{\beta}

\theoremstyle{plain}

\newtheorem{thm}{Theorem}[section]

\newtheorem{prop}[thm]{Proposition}

\newtheorem*{conj*}{Conjecture}

\newtheorem*{cor*}{Corollary}

\theoremstyle{definition}

\newtheorem*{rem*}{Remark}

\newtheorem*{rems*}{Remarks}

\newtheorem{example}{Example}[section]

\newcommand{\GITl}[1]{\backslash \!\! \backslash _{\kern-.2em #1 \kern0.1em}}
\newcommand{\GIT}[1]{/\!\!/_{\kern-.2em #1 \kern0.1em}}

\renewcommand{\l}{\left}
\renewcommand{\r}{\right}
\newcommand{\bra}{\left\langle}
\newcommand{\ket}{\right\rangle}

\newcommand{\ev}{\operatorname{ev}}
\newcommand{\qbinom}{\genfrac{[}{]}{0pt}{}}

\def\bred{\begin{color}{red}}
\def\ered{\end{color}}

\def\bes{\begin{subequations}}
\def\ees{\end{subequations}}

\usepackage{tikz}
\usetikzlibrary{decorations.pathreplacing,angles,quotes}

\newcommand\PP{\mathbb P}

\newcommand\Z{\mathbb Z}

\newtheorem{lemma-definition}[theorem]{Lemma-Definition}

\theoremstyle{definition}

\theoremstyle{remark}

\numberwithin{equation}{section}
\numberwithin{figure}{section}

\newcommand{\ZZ} {\mathbb{Z}}

\renewcommand{\AA} {\mathbb{A}}

\newcommand {\shC}  {\mathcal{C}}

\newcommand {\shE}  {\mathcal{E}}

\newcommand {\shO}  {\mathcal{O}}

\newcommand {\shY}  {\mathcal{Y}}

\newcommand {\ol} {\overline}

\DeclareMathOperator {\hhh} {H}

\newcommand{\sstyle}{\scriptstyle}

\def\mydate{\ifcase\month \or January\or February\or March\or
April\or May\or June\or July\or August\or September\or October\or 
November\or December\fi \space\number\day,\space\number\year}

\DeclareMathOperator{\mmm}{M}

\DeclareMathOperator{\scatt}{Scatt}

\newcommand*\circleed[1]{\tikz[baseline=(char.base)]{
            \node[shape=circle,draw,inner sep=1pt] (char) {#1};}}

\newcommand{\exI}{\bbP(1,a,b)}
\newcommand{\exII}{Y^{[2]}_{(a,b)}}
\newcommand{\exIII}{Y^{[3]}_{(a,b)}}

\begin{document}

\title{
  Stable maps to Looijenga pairs: orbifold examples
  }

\author{Pierrick Bousseau}

\address{\tiny
Universit\'e Paris-Saclay, CNRS, Laboratoire de mathématiques d'Orsay, 91405, Orsay, France}

\email{pierrick.bousseau@u-psud.fr}



\author{Andrea Brini}
\address{\tiny School of Mathematics and Statistics,
University of Sheffield, S11 9DW, Sheffield, United Kingdom. On leave from CNRS, DR 13, Montpellier, France}
\email{a.brini@sheffield.ac.uk}

\author{Michel van Garrel}

\address{\tiny University of Birmingham, School of Mathematics, B15 2TT, Birmingham, United Kingdom}
\email{m.vangarrel@bham.ac.uk}

\thanks{This project has been supported by the European Union's Horizon
  2020 research and innovation programme under the Marie Sklodowska-Curie
  grant agreement No 746554 (M.~vG.), the Engineering and Physical Sciences
  Research Council under grant agreement ref.~EP/S003657/2 (A.~B.) and by Dr.\ Max R\"ossler, the Walter Haefner Foundation and the ETH Z\"urich Foundation (P.~B. and M.~vG.).}

\begin{abstract}
In \cite{BBvG2} we established a series of correspondences relating five enumerative theories of log Calabi--Yau surfaces, i.e.\ pairs $(Y,D)$ with $Y$ a
smooth projective complex surface and $D=D_1+\dots +D_l$ an anticanonical divisor on $Y$ with each $D_i$ smooth and nef.
In this paper we explore the generalisation to $Y$ being a smooth Deligne--Mumford stack with projective coarse moduli space of dimension 2, and $D_i$ nef $\bbQ$-Cartier divisors.
We consider in particular three infinite families of orbifold log Calabi-Yau
surfaces, and for each of them we provide closed form solutions of 
the maximal contact log Gromov--Witten theory of the pair $(Y,D)$, the local Gromov--Witten theory of the total space of $\bigoplus_i \cO_Y(-D_i)$, and the open Gromov--Witten of toric orbi-branes in a Calabi--Yau 3-orbifold associated to $(Y,D)$. We also consider new examples of  BPS integral structures underlying these invariants, and relate them to the Donaldson--Thomas theory of a symmetric quiver specified by $(Y,D)$, and to a class of open/closed BPS invariants.
\end{abstract}

\dedicatory{In grateful memory of Boris Anatol'evich Dubrovin, 1950--2019}

\maketitle
\setcounter{tocdepth}{1}
\tableofcontents

\section{Introduction}

In \cite{BBvG2}, we established a  series of correspondences between {\it a priori} distinct enumerative theories of Gromov--Witten (GW)/Donaldson--Thomas (DT) type associated to smooth log Calabi-Yau surface of maximal boundary with nef boundary components, or {\it nef Looijenga pairs}: these are pairs $(Y,D)$ where $Y$ is a smooth projective surface and $|-K_Y|\ni D=D_1+\cdots+D_l$ is an anticanonical normal crossings divisor with $l>1$ smooth and nef irreducible components $D_j$. For a nef Looijenga pair we proved an equivalence between the log GW theory of the pair $(Y,D)$, the local GW theory of the total space of the sum of dual line bundles to the irreducible components $D_j$, the open GW theory of Aganagic--Vafa branes in a Calabi--Yau threefold associated to $(Y,D)$, the DT theory of a symmetric quiver specified by $(Y,D)$, and a variety of BPS invariants considered by Klemm--Pandharipande, Ionel--Parker, and Labastida--Mari\~no--Ooguri--Vafa. 
Moreover, we provided  closed-form solutions for the calculation of the these invariants in all the finitely many deformation families of such pairs. \\

In this companion note we explore the extension of such correspondences to include orbifolds, and provide compelling evidence that the bulk of the correspondences put forward in \cite{BBvG2} generalise to this setting essentially verbatim.
We consider pairs $(\cY,D=D_1+\cdots +D_l)$ 
where $\cY$ is a smooth complex Deligne--Mumford stack with coarse moduli space a normal Gorenstein projective surface $Y$, $(Y,D)$ is log smooth (in particular, the singularities are concentrated along the codimension 2 strata of $D$), $D\in|-K_Y|$, and the irreducible components $D_j$ are nef and $\bbQ$-Cartier for all $j=1, \dots, l$. In particular we will exemplify how and to what extent our circle of correspondences in \cite{BBvG2} generalises to this context in three infinite families of log Calabi--Yau orbifolds:\\

\begin{description}[leftmargin=.4cm]
\item[Example I] in this example, $Y$ is taken to be  the weighted projective plane $\bbP(1,a,b)$ for $a,b$ positive coprime integers with fan given in Figure \ref{fan:P1ab}. This surface has two orbifold singularities that are locally quotients of $\bbC^2$ by the finite cyclic group $\mu_a$, resp.\ $\mu_b$, and there is a toric line $D_{(b,a)}$ that joins both. Extending  $D_{(b,a)}$ to an anticanonical divisor by adding a general member $D_2$ of $|-K_Y-D_{(b,a)}|$ gives the Looijenga orbi-pair $\exI \coloneqq (\bbP(1,a,b),D=D_{(b,a)}+D_2)$. It is non-toric since the topological Euler characteristic of the complement of $D$ is $\chi(\bbP(1,a,b)\setminus D)=1$. 
\item[Example II] in this case we blow up $\exI$
in a smooth point of $D_2$. We denote the resulting surface with its choice of 2-component boundary by $\exII$.
\item[Example III] blowing up $\bbP(1,a,b)$ in a smooth point of one of its toric divisors leads to a non-toric nef orbi-Looijenga pair with $l=3$, which we denote by $\exIII$.
\end{description}

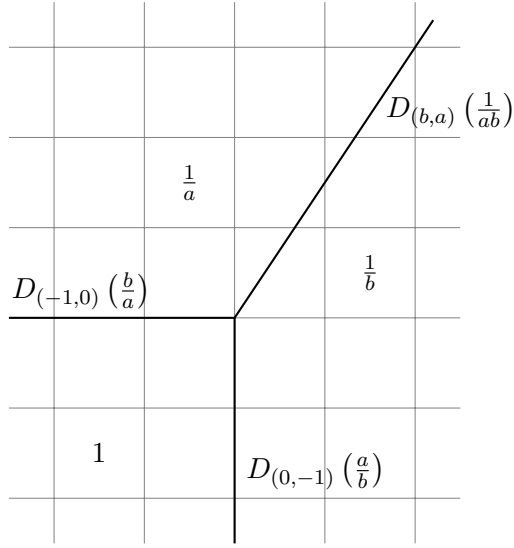
\begin{figure}[h]
\begin{tikzpicture}[smooth, scale=1.2]
\draw[step=1cm,gray,very thin] (-2.5,-2.5) grid (2.5,3.5);
\draw[thick] (-2.5,0) to (0,0);
\draw[thick] (0,-2.5) to (0,0);
\draw[thick] (0,0) to (2.2,3.3);
\node at (0.9,-1.7) {$D_{(0,-1)}\left(\frac{a}{b}\right)$};
\node at (2.4,2.3) {$D_{(b,a)}\left(\frac{1}{ab}\right)$};
\node at (-1.7,0.3) {$D_{(-1,0)}\left(\frac{b}{a}\right)$};
\node at (-1.5,-1.5) {$1$};
\node at (1.5,0.5) {$\frac{1}{b}$};
\node at (-0.5,1.5) {$\frac{1}{a}$};
\end{tikzpicture}
\caption{The fan of $\bbP(1,a,b)$ with toric intersection numbers}
\label{fan:P1ab}
\end{figure}

\section{Setup and main results}

\subsection{The enumerative theories}

Let $(Y,D)$ be a log smooth log Calabi-Yau surface with $D=D_1+\cdots+D_l$ and each $D_j$ irreducible and let $d\in\hhh_2(Y,\bbZ)$.
Provided that \'etale-locally around each singularity $Y(D)$ is isomorphic to a toric variety with its toric boundary, $Y(D)$ is log smooth. In particular this applies to Examples I-III.
We will use the short-hand notation $Y(D)$ to denote the log-scheme obtained by taking the divisorial log structure induced by $D$ on $Y$. 
For $n\geq0$, denote by $[n]_q$ the $q$-number $q^{\frac{n}{2}}-q^{-\frac{n}{2}}$, as well as the symmetrised $q$-factorials $[n]_q! \coloneqq \prod_{i=1}^n [i]_q$ and $q$-binomials $\qbinom{n}{m}_q \coloneqq [n]_q!/([m]_q! [n-m]_q!)$.

\subsubsection{All genus log GW invariants}

Let $g\geq0$. We are (virtually) counting genus $g$ degree $d$ curves in $X$ that have prescribed tangency conditions along the boundary $D$, namely we require the curves to meet each of $D_j$ in one point of maximal tangency $d\cdot D_j$. This is a moduli problem of virtual dimension $g+l-1$.

Log smoothness of $Y(D)$ guarantees the existence of the corresponding moduli space of basic stable log maps $\ol{\mmm}^{\log}_{g,m}(Y(D),d)$ as constructed by Abramovich--Chen \cites{Chen14,AbramChen14}
and Gross--Siebert \cite{GS13} by encoding the tangency conditions via log structures. It admits a rank $g$ vector bundle $\mathbb{E}$ whose fiber over $f \colon C \rightarrow Y(D)$ is the vector space $\hhh^0(C,\omega_C)$ of sections of the dualising sheaf of the domain curve. To cut down the virtual dimension to 0, we require the curves to pass through $l-1$ general points in the interior $Y\setminus D$ and further cap the virtual fundamental class by the top Chern class $\lambda_g:=c_g(\mathbb{E})$ of $\mathbb{E}$, leading to invariants
\begin{equation} \label{eq_log}
N^{\rm log}_{g,d}(Y(D)):=\int_{[\ol{\mmm}^{\log}_{g,l-1}(Y(D),d)]^{\rm vir}} (-1)^g \, \lambda_g \, \prod_{j=1}^{l-1} \ev_j^*([{\rm pt}]),
\end{equation}
where $\ev_j:\ol{\mmm}^{\log}_{g,l-1}(Y(D),d)\to X$ is the morphism given by evaluation at the $j$th point. We denote by $N^{\rm log}_{d}(Y(D)):=N^{\rm log}_{0,d}(Y(D))$.

We package the invariants into the fixed-degree, all-genus generating function
\beq
\mathsf{N}^{\rm log}_d(Y(D))(\hbar) \coloneqq \frac{1}{\left( 
2 \sin \left( \frac{\hbar}{2} \right) \right)^{l-2}}
\sum_{g \geqslant 0} N^{\rm log}_{g,d} \hbar^{2g-2+l} \,.
\eeq
By a combination of \cites{MR3904449,bousseau2018quantum} (see  \cite[Proposition 4.2]{BBvG2}), after the change of variable $q=e^{i \hbar}$, $\mathsf{N}^{\rm log}_d(Y(D))(\hbar)$ is the power series expansion in $\hbar$ of polynomials
$\mathsf{N}^{\rm log}_d(Y(D))(q)$ in $q^{\frac{1}{2}}$. 

The invariants 
$N^{\rm log}_{g,d}(Y(D))$ can be naturally viewed as invariants of the $3$-fold $Y(D) \times \AA^1$. Indeed, the Gromov--Witten obstruction theories for stable maps to the surface $Y(D)$ and the $3$-fold 
$Y(D)\times \AA^1$ differ by the space $H^1(C,\mathcal{O}_C)$
\cite[Lemma 7]{MPT}, which by Serre duality is dual to $H^0(C,\omega_C)$.
Thus, the insertion of $(-1)^g \lambda_g$ in the Gromov--Witten theory of the surface $Y(D)$ exactly reproduces Gromov--Witten invariants of the $3$-fold $Y(D) \times \AA^1$. Essentially for the same reason, the class 
$(-1)^g \lambda_g$ appears also naturally in the higher genus extension of the log-local correspondence for a smooth divisor presented in \cite{Bousseau:2020ckw}. In \cites{BBvG2}, in particular Section 1.4, we explained how higher genus log Gromov--Witten invariants of log Calabi-Yau surfaces with insertion of $(-1)^g \lambda_g$ fit into a web of 
log-local-open correspondences. In the present paper, we use the invariants $N^{\rm log}_{g,d}(Y(D))$ to extend this web of correspondences to the orbifold setting.

\subsubsection{Local GW invariants}
A different class of invariants of $(Y,D)$ arises by considering the local theory of $\mathrm{Tot}\big( \bigoplus_j \cO_Y(-D_j)\big)$. This is a non-compact Calabi--Yau-$(l+2)$ fold, and since for $l>1$ there are no non-zero Gromov--Witten invariants with point insertions for $g >0$, we restrict below to genus 0.\\
Suppose $Y$ admits a presentation as a Gorenstein GIT quotient $Y=Z\GIT{}{G}$ for a complex smooth projective variety  $Z$ and reductive group $G$, and write $\cY=[Z\GIT{}{G}]$ for the Deligne--Mumford quotient stack it represents. While  the Gorenstein surface $Y$ (with trivial log structure) is not smooth, $Y(D)$ and $\cY$ are smooth in the respective categories by definition.
The genus 0 local GW theory of $Y(D)$ is a (virtual) count of rational orbi-curves in the $(l+2)$-dimensional non-compact Calabi--Yau orbifold $\mathcal{E}_{Y(D)} \coloneqq \mathrm{Tot}(\oplus_{i=1}^l (\cO_{\shY}(-D_i)))$ with coarse space $E_{Y(D)}\coloneqq \mathrm{Tot}(\oplus_{i=1}^l (\cO_{Y}(-D_i)))$. Given that the log GW theory is insensitive to the twisted sector, we will only be interested in the untwisted sector of the orbifold GW theory of $\shE_{Y(D)}$ \cite{MR2450211}.

Denote by $\ol{\mmm}_{0,m}(\shY,d)$ the moduli stack of twisted genus 0 $m$-marked stable maps $[f:\shC\to\shY]$ with $f_*([\shC])=d$ and $\shC$ an $m$-pointed twisted curve. We write $\ol{\mmm}_{0,m}(Y,d)$ for the substack of twisted stable maps such that the image of the evaluation maps is contained in the age zero component of the inertia stack of $\shY$.
The moduli stack $\ol{\mmm}_{0,m}(\shY,d)$ has a perfect obstruction theory, inducing a virtual fundamental class
\beq
[\ol{\mmm}_{0,m}(Y,d)]^{\rm vir}\in\hhh_{2{\rm vdim}}(\ol{\mmm}_{0,m}(Y,d),\bbQ),
\eeq
where ${\rm vdim} = -K_Y\cdot d + m - 1$.

Assuming that $d\cdot D_j>0$, there is a rank $-K_Y\cdot d-l$ obstruction vector bundle ${\rm Ob}_{D}$ on $\ol{\mmm}^{\log}_{0,m}(Y,d)$ with fibre $\hhh^1(\shC,f^*\bigoplus_{j=1}^l\shO_{\shY}(-D_j))$ over a twisted stable map $[f:\shC\to\shY]$.
It is defined as ${\rm Ob}_D \coloneqq R^1\pi_*(f^{\rm univ})^*\left( \bigoplus_{j=1}^l \shO_Y(-D_j) \right)$ for $\pi : \shC^{\rm univ} \to \ol{\mmm}_{0,m}(\shY,d)$ the universal curve and $f^{\rm univ} : \shC^{\rm univ} \to \shY$ the universal twisted stable map.
Restricting to the component of the inertia stack of age zero, we obtain the virtual fundamental class
\[
[\ol{\mmm}_{0,m}(E_{Y(D)},d)]^{\rm vir}:=[\ol{\mmm}_{0,m}(Y,d)]^{\rm vir}\cap c_{\rm top}({\rm Ob}_D)\in \hhh_{2(l-1+m)}(\ol{\mmm}_{0,m}(Y,d),\bbQ).
\]

Restricting to the untwisted sector yields evaluations maps ${\rm ev}_j : \ol{\mmm}_{0,m}(Y,d)\to Y$ and
we define the following two classes of local invariants
\bea
\label{eq:Nloc}
N^{\rm loc}_{d}(Y(D)) & \coloneqq & \int_{[\ol{\mmm}_{0,l-1}(E_{Y(D)},d)]^{\rm vir}} \, \prod_{j=1}^{l-1} \ev_1^*([{\rm pt}]) \,, \\
N^{\rm loc, \psi}_{d}(Y(D)) & \coloneqq & \int_{[\ol{\mmm}_{0,1}(E_{Y(D)},d)]^{\rm vir}} \, \ev_1^*([{\rm pt}]) \psi_1^{l-2} \,,
\label{eq:Nlocpsi}
\eea 
where $\psi_i=c_1(\bbL_i)$ is the first Chern class of the $i^{\rm th}$ tautological line bundle on $\ol{\mmm}_{0,m}(Y,d)$.

\subsubsection{All genus open GW invariants}

In \cite[Construction 6.4]{BBvG2}, we showed how to associate to a smooth Looijenga pair $Y(D)$ satisfying certain positivity properties \cite[Definition 6.3]{BBvG2} a triple $Y^{\rm op}(D)=(X,L,\mathsf{f})$ with
$X$ a semi-projective toric Calabi--Yau 3-fold, 
$L=L_1 \cup \cdots \cup L_{l-1}$ a disjoint union of 
$l-1$ Aganagic--Vafa toric  Lagrangians \cite{Aganagic:2000gs} in $X$ and $\mathsf{f}$ a framing for $L$. At first approximation, $X$ is the total space of $K_{Y\setminus\left(D_1\cup\dots\cup D_{l-1}\right)}$, the $L_j\simeq S^1\times\bbR^2$ are Lagrangians that contract to a vanishing cycle $[S^1]$ of $Y$ near $D_j$, and $\mathsf{f}$ is determined by the compactification given by adding back the $D_j$, $j<l$.
See \cite[Construction 6.4]{BBvG2}, the framing determines the compactification of (a toric limit of) $Y\setminus(D_1\cup\cdots\cup D_{l-1})$ to (a toric limit of) $Y$. At the level of their associated polyhedra, the framing determines the additional halfspaces to intersect with to go from from the polyhedron of $Y\setminus(D_1\cup\cdots\cup D_{l-1})$ to the polytope of $Y$ (with anticanonical polarizations). The framings correspond to the slopes of the edges.

It is immediate to verify from \cite[Construction 6.4]{BBvG2} that the above generalises to the case of Looijenga orbi-pairs $Y(D)$, for which $Y^{\rm op}(D)=(X,L,\mathsf{f})$ is in general a semi-projective Gorenstein orbifold $X$ with fractionally framed orbifold toric Lagrangians $(L,\mathsf{f})$ \cite{MR2861610}. 
The orbifold case introduces a small modification. If $Y$ has orbifold singularities at the toric 0-strata, we obtain rational framings.
More precisely, denote by $\bbA^1$ one of the toric strata corresponding to an outer edge of the toric graph with framing. Then the $\bbA^1$ is compactified in $Y$ by adding a point in $Y$ that is a cyclic quotient singularity with isotropy group $\bbZ/r\bbZ$, for $r$ the denominator of the framing.
Adapting it to the orbifold case, the construction moreover induces a natural injection $\iota \colon \hhh^{\rm rel}_2(Y^{\rm op}(D),L;\Z) \hookrightarrow {\rm A}_1(Y,\Z)$ as we review in Section \ref{sec:logopen}, with all the curve classes $d\in\hhh_2(Y,\bbZ)$ lying in its image. 
In Examples I--III, $X$ and $L$ will always be smooth, but the framing $\mathsf{f}$ will be fractionally shifted by rational numbers $f_i=p_i/r_i$ from the canonical framing on each connected component $L_i$ of $L$.

The open GW theory of $Y^{\rm op}(D)$ was defined in the algebraic category\footnote{See \cite{MR2861610} for a definition of open GW invariants of toric orbifold Lagrangians using localisation, and \cite{MR3948935} for a definition for smooth toric Lagrangians with fractional framing using relative GW theory.} in \cite{Li:2004uf}. Given partitions $\mu_i$ of length $\ell(\mu_i)$, $i=1, \dots, l-1$, there is a virtual dimension zero moduli space $\ol{\mmm}_{g; \beta; \mu_1, \dots, \mu_{l-1}}(Y^{\rm op}(D))$ of relative degree $\beta$ open stable morphisms to $Y^{\rm op}$ from genus-$g$, open Riemann surfaces with $\sum_{i=1}^{l-1} \ell(\mu_i)$ connected components of the boundary mapping to $L_i$ with winding numbers around $S^1 \hookrightarrow L_i$ equal to the parts of $\mu_i$. The corresponding open GW invariants,
\beq
O_{g,\beta,\vec \mu}(Y^{\rm op}(D)) = \int_{[\ol{\mmm}_{g; \beta; \vec \mu} (Y^{\rm op}(D))]^{\rm vir}} 1
\eeq 
can be encoded into formal generating functions
\bea
\mathsf{O}_{\beta;\vec \mu}(Y^{\rm op}(D))(\hbar) &\coloneqq & \sum_{g} \hbar^{2g-2+\ell(\vec\mu)}  O_{g; \beta; \vec \mu}(Y^{\rm op}(D)) \,,
\label{eq:openfreeen}
\eea
with $\ell(\vec\mu)=\sum_{i=1}^s \ell(\mu_i)$. We will write simply $O_{g; \beta}(Y^{\rm op}(D))$ and $\mathsf{O}_{\beta}(Y^{\rm op}(D))(\hbar)$ 
for the $(l-1)$-holed open GW invariants obtained when 
$\mu_i=(m_i)$, which are then determined by 
the class $\beta  \in \hhh^{\rm rel}_2(Y^{\rm op}(D),L;\ZZ)$. 

\subsubsection{Quiver DT invariants}

Let $\mathsf{Q}$ be a symmetric quiver with $n$~vertices and, for dimension vectors $\mathsf{d}=\sum_i d_i v_i$, $\mathsf{e}=\sum_i e_i v_i \in \bbN \mathsf{Q}_0= \bbN v_1+ \dots+\bbN v_n $, denote by $E_\mathsf{Q}(\mathsf{d},\mathsf{e})$ the Gram matrix of the Euler form
\beq
E_\mathsf{Q}(\mathsf{d},\mathsf{e}) \coloneqq \sum_{i=1}^n d_i e_i -\sum_{\a:v_i \to v_j} d_i e_j\,.
\eeq
The motivic DT invariants $\mathrm{DT}_{\mathsf{d}; i}(\mathsf{Q})$ of $\mathsf{Q}$ are defined from the plethystic generating function
\beq
\mathrm{Exp}\l(\frac{1}{[1]_q} \sum_{\mathsf{d} \neq 0} \sum_{i \in \bbZ} \mathrm{DT}_{\mathsf{d}; i}(\mathsf{Q}) \mathsf{x}^\mathsf{d} (-q^{1/2})^{-i} \r) = 
\sum_{\mathsf{d} \in \bbN^n} \frac{\big(-q^{1/2}\big)^{E_Q(\mathsf{d},\mathsf{d})} \mathsf{x}^\mathsf{d}}{\prod_{i=1}^n (q;q)_{d_i}} \,,
\label{eq:DTmot_intro}
\eeq
where $\mathsf{x}^\mathsf{d} =\prod_{i=1}^n x_i^{d_i}$. 
Using the terminology of 
\cite[\S 3.3]{Bou19a}, the right-hand side is the generating series of Poincar\'e rational functions of the stacks of representations of $\mathsf{Q}$. The numerical DT invariants $\mathrm{DT}^{\rm num}_{\mathsf{d}}(\mathsf{Q})$
are non-negative  \cite{MR2956038} integers defined by
\beq
\mathrm{DT}^{\rm num}_{\mathsf{d}}(\mathsf{Q}) \coloneqq \sum_{i \in \bbZ} (-1)^i \mathrm{DT}_{\mathsf{d},i}(\mathsf{Q}) \,.
\label{eq:DTnum}
\eeq

\subsubsection{BPS invariants}
For a Looijenga orbi-pair $Y(D)$, we define open BPS numbers as in \cite[Equation (1.21)]{BBvG2} by
\beq
 \Omega_{d}(Y(D))(q) \coloneqq  [1]_{q}^2 \left( \prod_{i=1}^{l} \frac{1}{[ d \cdot D_i]_{q}} \right)
\sum_{k | d}
 \frac{(-1)^{ d/k \cdot D + l} \mu(k)}{[k]_{q}^{2-l} k^{2-l}} \,
\mathsf{N}_{d/k}^{\rm log}(Y(D))(-\ri k \log q)\,.
\label{eq:Omegad}
\eeq
%
%
We will also denote just by $\Omega_{d}(Y(D))$ the genus-zero limit $\Omega_{d}(Y(D))(1)$. 

The log-open correspondence of \cref{thm:logopen} below implies that, for Examples I--III,
\beq
 \Omega_{d}(Y(D))(q) =
[1]_q^{2} \prod_{i=1}^{l-1} \frac{r_i (d\cdot D_i)}{[d\cdot D_i]_q}  \sum_{k | d}  \frac{\mu(k) (-1)^{\sum_{i=1}^{l-1} d/k\cdot D_i (r_i+1)}}{k}
\mathsf{O}_{\iota^{-1}(d/k)}(Y^{\rm op}(D))(-\ri k \log q)\,,
\label{eq:Omegad2}
\eeq 
where $f_i=p_i/r_i$ with $(p_i, r_i)=1$ is the framing of the $i^{\rm th}$ orbifold Aganagic--Vafa Lagrangian in $Y^{\rm op}(D)$. Even though $\Omega_d(q)$ can at most be expected to be a rational function of $q^{1/2}$, heuristically, and for smooth, integrally framed $Y^{\rm op}(D)$ \cites{Labastida:2000yw, Labastida:2000zp, Ooguri:1999bv,Marino:2001re}, $\Omega_d(Y(D))(q)$ has an interpretation as generating function of BPS domain walls counts in a type IIA string compactification on $Y^{\rm op}(D)$, with its coefficient computing degeneracies of D2-branes with fixed spin and charge ending on a D4-brane wrapped around the Lagrangians $Y^{\rm op}(D)$. The formula \eqref{eq:Omegad2} generalises \cite[Eq.~2.10]{Marino:2001re} to the orbifold setting, with an additional factor keeping track of the fractional framing of the branes.

\subsection{The correspondences}

In our previous paper \cite{BBvG2} we proposed that the invariants of the previous Section are related through a series of geometric correspondences.
A conceptual explanation of these was provided in \cite[Section 1.4]{BBvG2}, and we briefly recall it in \cref{sec:motivation}.

\subsubsection{Numerical log-local}

Our first result is the following

\begin{thm}\label{thm:log-local}
Assume that $Y(D)$ is one of $\exI$, $\exII$ or $\exIII$. Then
\beq
N^{\rm loc}_d(Y(D))=\left( \prod_{j=1}^l \frac{(-1)^{d\cdot D_j -1}}{d\cdot D_j} \right) N^{\rm log}_d(Y(D))\,.
\label{eq:logloc}
\eeq
\label{thm:logloc}
\end{thm}

Our proof of \cref{thm:logloc} follows from a stronger result, wherein we give a complete closed-form solution of both sides of \eqref{eq:logloc} in all degrees. In the case of an irreducible smooth nef divisor, the correspondence between genus 0 log and local GW invariants was proven in all dimensions at the cycle-level in \cite{vGGR}, with various extensions in \cites{BBvG1,BBvG2,GWZ, CGKT1,CGKT2,CGKT3,Bou19a,Bou19b,NR,FTY,tseng2020mirror,BNTY}.
The naive conjectural extension of this log-local correspondence at the cycle level for normal crossings divisors has been recently disproved \cite{NR, BNTY}. However, the numerical version of the log-local correspondence for normal crossing divisors seems to hold in a number of cases of great interest: for example this was proved for point insertions of orbifold toric pairs in \cite{BBvG1}, and for point invariants of log Calabi--Yau surfaces with nef $D_i$ in 
\cite{BBvG2}. \cref{thm:logloc} simultaneously provides a non-toric, orbifold version of the numerical version of the log-local correspondence of \cites{vGGR,BBvG1}.\footnote{As discussed in more details in 
\cite[\S 1.4]{BBvG2}, point insertions and the log Calabi--Yau condition are both crucial assumptions allowing us to obtain the numerical log-local correspondence despite the general negative results of \cite{NR, BNTY}. In \cite[\S 5]{BBvG2}, we gave a conceptual proof by degeneration of the numerical log-local correspondence for log Calabi-Yau surfaces with two boundary components. This completely general degeneration argument can
be easily extended to the orbifold setting in any dimension and can in
principle be used to determine when the numerical log-local correspondence holds and when correction terms are needed.}

\subsubsection{Log-open}\label{sec:logopen}

Our second result is an orbifold generalisation of the higher genus log-open principle of \cite[Conjecture 1.3]{BBvG2}. Following \cite[Definition 6.5]{BBvG2}, we canonically identify each curve degree $d\in\hhh_2(Y,\bbZ)$ with a relative curve degree $\iota^{-1}(d)$ in $Y^{\rm op}(D)$. We recall how this identification works and adapt it to the orbifold setting.
The class of a Riemann surface with boundary in $Y^{\rm (op)}(D)$ is decomposed as an $l$-tuple $(\beta,\alpha_1,\dots,\alpha_{l-1})$. Here $\beta$ is a 2-homology class, which decomposes as a sum of the homological 2-spheres corresponding to the compact toric 1-strata of $Y^{\rm op}(D)$ (the inner edges of the toric graph). The $\alpha_i$ are relative 2-homology classes corresponding to the outer edges with framing.
The morphisms
\beq 
K_{Y \setminus D_1 \cup \dots \cup D_{l-1}} \stackrel{\pi}{\longrightarrow}  Y \setminus D_1 \cup \dots \cup D_{l-1} \stackrel{i}{\longrightarrow}    Y\,,
\eeq 
with $\pi:K_{Y \setminus D_1 \cup \dots \cup D_{l-1}} \to Y \setminus D_1 \cup \dots \cup D_{l-1}$ the bundle projection and $i:Y \setminus D_1 \cup \dots \cup D_{l-1} \hookrightarrow Y$ the canonical open immersion, induce an injective homomorphism of 2-homology groups 
\beq\iota_{o}: \hhh_2(Y \setminus D_1 \cup \dots \cup D_{l-1}, \bbZ) \hookrightarrow \hhh_2(Y, \bbZ)\,.
\eeq
In fact, by sending the generators to the corresponding subvarieties, we identify $\iota_{o}$ with a morphism
\beq
\hhh_2(Y \setminus D_1 \cup \dots \cup D_{l-1}, \bbZ) \hookrightarrow {\rm A}_1(Y)\,.
\eeq
The edge of the toric graph with framing $p_i/r_i$ corresponds to a toric $(\bbA^1)_i$ in $K_{Y \setminus D_1 \cup \dots \cup D_{l-1}}$ that meets the $i^{\rm th}$ connected component $L_i$ of the toric Lagrangian $L$ in a non-trivial minimal $(S^1)_i\subset L_i$, i.e.~$[(S^1)_i]$ generates $\hhh_1(L_i, \bbZ)$. Moreover, $[(S^1)_i]$ corresponds to the relative homology class
\[
[({\rm disk} \subset (\bbA^1)_i,\partial \, {\rm disk}=(S^1)_i)] \in \hhh_2^{\rm rel}(Y^{\rm op}(D),L;\bbZ).
\]
The projection $\pi((\bbA^1)_i)$ is compactified to an orbifold $\bbP^1_{r_i}$ in $Y$, with the added point a cyclic quotient singularity in $Y$ with isotropy group $\bbZ/r_i\bbZ$, see \cite[Section 3.11.3]{MR3948935}. Furthermore, $r_i [\bbP^1_{r_i}]$ is a well-defined class in $\hhh_2(Y, \bbZ)$. 
Following \cite[Definition 6.5]{BBvG2}, we send $[(S^1)_i]$ to $[\bbP^1_{r_i}]$. The latter is not in $\hhh_2(Y,\bbZ)$, but it is in the Chow group ${\rm A}_1(Y)$. We will only be interested in winding numbers whose image lies in $\hhh_2(Y,\bbZ)$. In particular, in the log-open correspondence of \eqref{eq:log-open}, only winding numbers that are multiples of the $r_i$ of their respective framings can be compared with curves classes on the log side, which introduces the correction term in \eqref{eq:log-open} compared to \cite[Conjecture 1.3]{BBvG2}.
In summary, we extend $\iota_o$ to an injective map 
\beq 
\iota \colon \hhh_2(Y \setminus D_1 \cup \dots \cup D_{l-1}, \bbZ) \oplus \bigoplus_{i=1}^{l-1} \hhh_1\l(L_i, 
\bbZ\r) \hookrightarrow {\rm A}_1(Y)
\label{eq:iota}
\eeq 
by positing that
$\iota : [(S^1)_i] \mapsto [\bbP^1_{r_i}]$ for $i=1, \dots, l-1$, and we will only be interested in curve classes $d\in\hhh_2(Y,\bbZ)$ that admit an inverse $\iota^{-1}(d)$.

\begin{thm}
For each of $Y(D)=\exI$, $\exII$ or $\exIII$, writing $f_j=\frac{p_j}{r_j}$ with $(p_j,r_j)=1$ for the framing of the $j^{\rm th}$ Aganagic--Vafa orbi-brane in $Y^{\rm op}(D)$, we obtain
\beq\label{eq:log-open} 
 \mathsf{O}_{\iota^{-1}(d)}(Y^{\rm op}(D)) (-\ri \log q)
 = 
 [1]_q^{l-2} \,
  \frac{(-1)^{d \cdot D_l-1}}{[d \cdot D_l]_q} \,
  \prod_{j=1}^{l-1} \frac{(-1)^{r_j (d \cdot D_j)-1}}{r_j (d \cdot D_j)} \,
 \mathsf{N}_{d}^{\rm log}(Y(D)) (-\ri \log q)\,,
\eeq
as well as closed-form expressions of the invariants. The correction factor $r_j(d\cdot D_j)$ is the winding number around $L_j$.
\label{thm:logopen}
\end{thm}

In the genus zero limit ($q \to 1$) \cref{thm:logopen} recovers a version of the numerical log/local correspondence of \cref{thm:log-local}, with the genus zero open invariants equating the local invariants up to a factor:
\beq
\label{eq:loc-open} 
 O_{0;\iota^{-1}(d)}(Y^{\rm op}(D))
 =  \prod_{j=1}^{l-1} \frac{(-1)^{d\cdot D_j(r_j-1)}}{r_j}
 N_{d}^{\rm loc}(Y(D))\,.
\eeq
The additional normalisation factor as compared to the case of smooth varieties (where $r_j=1$), and especially the rescaling of the boundary circle classes by $r_j$ in the definition of $\iota$ are familiar in the relation of fractionally framed toric branes to enumerative invariants, and they match identically the correction factors relating open GW invariants of large $N$ Lagrangians of torus links from fractionally framed open GW invariants of toric orbi-branes; see \cites{Brini:2011wi,Diaconescu:2011xr,MR3948935,Aganagic:2013jpa}.

\subsubsection{KP/LMOV/DT integrality}

The next Theorem substantiates the expectation that \eqref{eq:Omegad}-\eqref{eq:Omegad2} are particular open BPS/LMOV partition functions, and in particular integral Laurent polynomials in $q^{1/2}$.
\begin{thm}
Let $Y(D)$ be any of $\exI$, $\exII$ or $\exIII$. Then $\Omega_d(q)\in \bbZ[q^{\pm 1/2}]$.
\label{thm:openbps}
\end{thm}
In the genus zero limit, the combination of \cref{thm:logloc,thm:logopen} gives
\bea
\Omega_{d}(Y(D))
& = & \frac{1}{\prod_{i=1}^{l} (d \cdot D_i)}
 \sum_{k | d}(-1)^{  d/k \cdot D + l} \frac{\mu(k)}{k^{4-2l}}
N_{d/k}^{\rm log}(Y(D)) \nn \\
& = & 
\label{eq:Omega0d}
\sum_{k | d} \frac{\mu(k)}{k^{4-l}} \prod_{j=1}^{l-1} r_j (-1)^{d/k \cdot D_j (r_j+1)}
O_{0;\iota^{-1}(d/k)}(Y^{\rm op}(D))  \\
& = & 
\sum_{k | d} \frac{ \mu(k)}{k^{4-l}}
N_{d/k}^{\rm loc}(Y(D))\,. 
\label{eq:KP}
\eea 
It follows directly from \cref{thm:openbps} that $\Omega_d(Y(D)) \in \bbZ$. In particular, \eqref{eq:KP} implies that the unrefined BPS invariants $\Omega_d(Y(D))$ coincide with an orbifold generalisation of the conjecturally integral  invariants $\mathrm{KP}_d(E_{Y(D)})$ of the local orbifold surfaces $\exI$, $\exII$ and $\exIII$ introduced by Klemm--Pandharipande for $l=2$ in \cite{Klemm:2007in} and by Ionel--Parker in \cite{MR3739228}  for smooth varieties. \cref{thm:openbps} implies then immediately their integrality in the generalised orbifold context of this paper.

Finally, when $l=2$ and the Lagrangians of $Y^{\rm op}(D)$ are integrally framed, the integrality statement above is also a consequence of the following statement, which is a direct consequence of the strips-quiver correspondence of \cite{Panfil:2018sis} (see in particular \cite[Theorem~7.3]{BBvG2})

\begin{thm}
Let $Y(D)=\exI$ or $Y(D)=\exII$ with either $a=1$ or $b=1$. Then there exists a symmetric quiver $\mathsf{Q}(Y(D))$ with $\chi(Y)-1$ vertices and a lattice isomorphism $\kappa: \bbZ (\mathsf{Q}(Y(D)))_0 \stackrel{\sim}{\rightarrow}  \hhh_2(Y,\bbZ) $ such that
\beq
\mathrm{DT}_{d}^{\rm num}(\mathsf{Q}(Y(D))) =
\Big|\Omega_{\kappa(d)}(Y(D))+ \sum_i \a_i \delta_{d,v_i}\Big| \,,
\label{eq:kpdt}
\eeq
with $\a_i \in \{0,1\}$. In particular, $\Omega_{d}(Y(D))\in \bbZ$.
\label{thm:kpdt}
\end{thm}

\subsubsection{Geometric motivation}
\label{sec:motivation}

In the smooth case, the rational underpinning of the web of correspondences of the previous Section was described in \cite[Section 1.4]{BBvG2}; we recall it briefly here. 
Let $Y(D=D_1+\cdots+D_l)$ be a log smooth log Calabi--Yau surface and let $d$ be a curve class such that $d\cdot D_l>0$. If $D_l$ is nef, or more generally if $d$ is $D_l$-convex \cite[Section 1.4.1]{BBvG2}, then by the main result of \cite{vGGR} the genus 0 log Gromov--Witten (GW) theory of maximal tangency of $(Y,D_l)$ is equivalent to the genus 0 local Gromov--Witten theory of ${\rm Tot} \l(\shO(-D_l) \to S\r)$. This correspondence extends to adding maximal tangency conditions along the divisors $D_j$ on $S$, resp.~$\shO(-D_l)|_{D_j}$ on ${\rm Tot} \l(\shO(-D_l) \to S\r)$: in other words, there is a duality between imposing a maximal tangency condition along $D_l$ and twisting by $\shO(-D_l)$. This insight forms the basis of many subsequent works \cites{BBvG1,BBvG2,GWZ, CGKT1,CGKT2,CGKT3,Bou19a,Bou19b,NR,FTY,tseng2020mirror,BNTY}. \\
An important question is the extent to which the above generalises to simple normal crossings divisors,
particularly in light of a counter-example for the log-local correspondence which was given in the log Fano case in \cite{NR}. On the other hand there is  growing evidence that the equivalence could persist in the log Calabi--Yau case in the stationary sector, see \cite[Theorem 3.4]{BBvG1}, \cite[Theorem 5.1]{BBvG2}, and the geometric argument of \cite[Section~5]{BBvG2} which is further amenable to treating the log-local correspondence for log smooth log Calabi--Yau varieties in any dimension. \cref{thm:log-local} further corroborates this expectation in the orbifold context. \medskip

The log-open correspondence of \cref{sec:logopen} stems from one of the realisations of \cite{BBvG2}, whereby a maximal tangency condition along $D_j$ is proposed to be heuristically replaceable by an open condition along a Lagrangian $L_j$ near $D_j$. For log~CY surfaces $(Y, D_1+\dots + D_l)$ this entails a precise correspondence at the level of logarithmic and open invariants, as we explain in \cite[Section~1.4.2]{BBvG2}: one  first employs the log-local correspondence above for the a single irreducible component $D_l$ to twist $Y$ by $\cO(-D_l)$, and then consider Lagrangians $L_j$ near $\shO(-D_l)|_{D_j}$, $j=1,\dots,l-1$. It is shown in \cite{BBvG2} that under suitable conditions these Lagrangians are singular Harvey--Lawson (Aganagic--Vafa) branes, for which a rigorous construction of the open invariants exists \cite{Li:2004uf}. The arguments of \cite[Section~1.4.2 and Construction~6.4]{BBvG2}, adapted to the orbifold context, lead then to the log-open relations \eqref{eq:log-open} and \eqref{eq:loc-open} above.\medskip



Finally, parallel to similar expectations for other enumerative theories, underlying the all genus Gromov--Witten invariants of $Y(D)$ are the BPS invariants defined by \eqref{eq:Omegad}. These are conjectured to be integer-valued Laurent polynomials in $q^{1/2}$, which we prove for Examples I-III in \cref{thm:openbps}. When $l=2$,
it was proposed in \cite[Section~1.4.3]{BBvG2} that the unrefined ($q\to1$) limit of these BPS invariants should recover, up to signs, the Donaldson--Thomas invariants of a symmetric quiver. This is suggested by the conjectured relation of Gopakumar--Vafa invariants with sheaf-counting  theories on 4-folds \cite{cao2019stable, cao2020stable} and their connection, for local surfaces, to moduli of quiver representations highlighted in \cite[Section~1.4.3]{BBvG2}, and simultaneously by  a joint use of the `branes-quivers' correspondence of \cite{Kucharski:2017ogk,Ekholm:2018eee} and the log-open correspondence above. It was proposed in \cite[Section~1.5.4]{BBvG2} that the same chain of ideas may apply in the orbifold context too: \cref{thm:kpdt} indeed establishes this expectation for the $l=2$ examples of this paper.

\subsection{The techniques}

\subsubsection{Scattering diagrams}
\label{sec:scattdiags}

Our main tool for the computations of $\mathsf{N}^{\rm log}_d(Y(D))$ are multiplications of quantum broken lines in the quantum scattering diagrams of \cites{mandel2015scattering, bousseau2018quantum, MR4048291, davison2019strong}.
In the classical limit, this is treated in \cites{GPS,GHKlog,Gro11} in dimension 2 and in full generality in \cites{GS11,GHS16}.
The quantum scattering diagram associated to $Y(D)$ consists of an affine integral manifold $B$ and a collection of walls $\mathfrak{d}$ with wall-crossing functions $f_{\mathfrak{d}}$.
Here we content ourselves to give a brief overview, referring to \cite[Section 4.2]{BBvG2} for details  of the construction, the wall-crossing algorithm, broken lines and their multiplication. In particular, \cite[Proposition 4.2]{BBvG2} explains how to extract the $\mathsf{N}^{\rm log}_d(Y(D))$ from structural coefficients of the multiplication of theta functions.

As a topological manifold, $B$ is homeomorphic to $\bbR^2$. It comes with some distinguished integral rays $\rho_1,\dots,\rho_{l+r}$ emanating from the origin. Up to reordering, $\rho_1,\dots,\rho_{l}$ will correspond to $D_1,\dots,D_l$. The collection of all rays forms a fan with associated toric variety $\overline{Y}(\overline{D}=D_{\rho_1}+\dots+D_{\rho_{l+r}})$ for the toric prime divisors $D_{\rho_1},\dots,D_{\rho_{l+r}}$ of $\ol{Y}$. Some of these rays, say $\rho_{j_1},\dots,\rho_{j_s}$ with possible repetitions, have focus-focus singularities on them. For our purposes, we perturb these away from their rays, which simply means the creation of lines parallel to the rays carrying the focus-focus singularities. These walls are decorated with wall-crossing functions. When two walls meet, there is scattering resulting in the creation of new walls carrying wall-crossing functions themselves. For our examples, there will only be ``simple'' scattering.

For each focus-focus singularity on a line parallel to $\rho_{j_i}$, we blow up a smooth point of $D_{\rho_{j_i}}$. Taken together, this yields a \emph{toric model}
\beq
\pi : (\widetilde{Y},\widetilde{D}) \longrightarrow (\ol{Y},\ol{D})\,.
\eeq
This means that there is a birational map 
\beq
\varphi : (\widetilde{Y},\widetilde{D}) \longrightarrow (Y,D)
\eeq
that is a sequence of blow ups at codimension 2 strata of the boundary. 
By \cite{AW}, $(\widetilde{Y},\widetilde{D})$ and $(Y,D)$ have the same log GW theory.

Then, $B$ has integral asymptotic directions that correspond to weighted blow ups of $(\overline{Y},\overline{D})$. Theta functions (which are sections of an ample line bundle on the mirror family) correspond to asymptotic directions. Their values on open subsets of the mirror corresponding to the chambers of the scattering diagram are given by the sums of the end-coefficients of broken lines coming from the corresponding asymptotic directions. The broken lines can bend when crossing walls picking up contributions from the wall-crossing functions.

Multiplying broken lines together corresponds to creating tropical curves with the correct weights (=intersection multiplicities) with a selection of (possibly weighted blow ups of) boundary divisors. The balancing condition has to be satisfied at each vertex except at the focus-focus singularities, which are seen as sources of Maslov index 0 disks. These tropical curves furthermore are weighted by contributions coming from the wall-crossing. Summing the weights of the tropical curves then calculates the $\mathsf{N}^{\rm log}_d(Y(D))$ as described in \cite[Proposition 4.2]{BBvG2}.

\subsubsection{Local mirror symmetry}
\label{sec:localtech}

Since $Y$ is a projective toric surface for all of Examples~I--III, we can avail ourselves of Givental-type mirror theorems to determine \eqref{eq:Nloc}-\eqref{eq:Nlocpsi}. Let $T \simeq (\bbC^\star)^l$ be the torus action on $E_{Y(D)}$ covering the trivial action on the zero section $\iota:Y \hookrightarrow E_{Y(D)}$. Fix $\{\varphi_\a\}_{\a=0}^{\chi(Y)-1} $ a $\hhh(BT)$-basis of $\hhh_T(E_{Y(D)})$ given by lifts to $T$-equivariant cohomology of classes $\phi_\a\in \hhh(Y)$ with $\deg \phi_\a \leq \deg\phi_{\a+1}$, and
for $\theta, \chi \in \hhh_T(E_{Y(D)})$ denote $\eta_{Y(D)}(\theta,\chi)$ the $T$-equivariant Poincar\'e pairing on $E_{Y(D)}$,
\beq
\eta_{Y(D)}(\theta,\chi) \coloneqq \int_Y \frac{\iota^* \theta \cup \iota^*\chi}{\cup_{m=1}^l \re_T(\cO_Y(-D_m))}\,.
\eeq 
In terms of the small $T$-equivariant $J$-function  of $E_{Y(D)}$, 
\beq
J_{\rm small}^{Y(D)}(t, z) \coloneqq z  \re^{\sum_{i=1}^{\rho(Y)} t_i \varphi_i/z}\l(1+\sum_{d\in
  \mathrm{NE}(Y)}\sum_{\a,\b} \eta_{Y(D)}^{-1}(\varphi_\a, \varphi_\b) \re^{t\cdot d} \bra \frac{\varphi_\a}{z(z-\psi_1)}
\ket_{0,1,d}^{E_{Y(D)}} \varphi_\b\r)\,,
\label{eq:Jsmall}
\eeq
\eqref{eq:Nlocpsi} is given by 
\beq
N^{\rm loc,\psi}_d(Y(D))= [z^{l-1} \re^{t \cdot d}] \eta_{Y(D)}\l(\mathrm{pt}, J_{\rm small}^{Y(D)}(t,z)\r),
\label{eq:NlocpsiJ}
\eeq 
where for a ring $R$ and $f(x_1, \dots, x_n)=\sum_{i_1,\dots, i_n\geq 0} c_{i_1\dots i_n} x_1^{i_1}\dots x_n^{i_n} \in R[[x_1, \dots, x_n]]$ a formal power series with $R$-coefficients, we write $[\prod_{l=1}^n x_{l}^{i_l}]f(x)$ for the formal Taylor coefficient $c_{i_1, \dots, i_n}$ in $(x_1, \dots, x_n)$, and 
we employed the usual correlator notation for GW invariants,
\beq
\bra \tau_1 \psi_1^{k_1}, \dots, \tau_n \psi^{k_n}_n\ket_{0,n,d}^{E_{Y(D)}}:=\int_{[\ol{\mmm}_{0,m}(E_{Y(D)},d)]^{\rm vir}} \prod_i \ev^*_i(\tau_i) \psi_i^{k_i}.
\eeq
We compute the r.h.s. of \eqref{eq:NlocpsiJ} using the Coates--Givental--Tseng twist \cites{MR1653024,MR1408320, MR2510741,MR2276766,MR2578300} at the $J$-function level, equating $J_{\rm small}^{Y(D)}(t,z)$, up to a mirror map $t\mapsto t(y)$, to an explicit generalised hypergeometric series $I^{Y(D)}(y,z)$, which is in turn read off from the fan of $E_{Y(D)}$ \cites{MR1653024,MR2510741,MR3414388,MR3412343}. In all of Examples~I--III we will have $t=\log y$, and thus
\beq 
N^{\rm loc,\psi}_d(Y(D))= [z^{l-1} y^d] \eta_{Y(D)}\l(\mathrm{pt}, I^{Y(D)}(y,z)\r).
\eeq 
To compute \eqref{eq:Nloc}, we use part of a reconstruction theorem due to Boris Dubrovin \cite[Lecture~6]{Dubrovin:1994hc}, combined with the vanishing of quantum corrections to certain products in quantum cohomology. Recall that the components of the big $J$-function,
\beq
J_{\rm big}^{Y(D)}(\tau, z) \coloneqq z+\tau+\sum_{\mathsf{d}\in
  \mathrm{NE}(Y)}
  \sum_{n\in \bbZ^+}\frac{1}{n!}
  \sum_{\a,\b} \eta_{Y(D)}^{-1}(\varphi_\a, \varphi_\b)
  \bra \tau, \dots, \tau, 
    \frac{\varphi_\a}{z-\psi_1}
\ket_{0,n+1,d}^{E_{Y(D)}} \varphi_\beta\,,
\eeq
form a basis of flat co-ordinates for the Dubrovin connection,
\beq 
z \nabla_{\theta}\nabla_{\chi} J_{\rm big}^{Y(D)}(\tau, z) = \nabla_{\theta \star_\tau \chi} J^{Y(D)}_{\rm big}(\tau, z)\,,
\eeq 
with $\tau \in \hhh_T(E_{Y(D)})$ and $\star_\tau$ the big quantum cohomology product, the restriction to small quantum cohomology being $\tau \to \sum_{i=1}^{\rho(Y)} t_i \varphi_i$. Suppose now that there exist numbers $\varpi^{Y(D)}_{ij}$, $i,j=1,\dots, \rho(Y)$ such that 
\beq
\sum_{i,j=1}^{\rho(Y)} \varpi^{Y(D)}_{ij} \varphi_i \star_\tau \varphi_j\Bigg|_{\tau \to \sum_{i=1}^{\rho(Y)} t_i \varphi_i} = \sum_{i,j=1}^{\rho(Y)} \varpi^{Y(D)}_{ij} \varphi_i \cup \varphi_j = \mathrm{pt}.
\eeq 
Then,
\beq 
z \sum_{i,j=1}^{\rho(Y)} \varpi^{Y(D)}_{ij} \partial^2_{t_i t_j} J_{\rm small}^{Y(D)}(t, z) = \nabla_{\mathrm{pt}} J^{Y(D)}_{\rm big}(\tau, z) \Bigg|_{\tau \to \sum_{i=1}^{\rho(Y)} t_i \varphi_i},
\eeq 
from which we deduce
\bea
N^{\rm loc}_d(Y(D)) &=& [z^{l-1} y^d] \eta_{Y(D)}\l(\mathrm{pt}, \sum_{i,j=1}^{\rho(Y)} \varpi^{Y(D)}_{ij} q_i q_j \partial^2_{q_i q_j} I^{Y(D)}(y,z)\r)\,, \nn \\
&=&  \l(\sum_{i,j=1}^{\rho(Y)} \varpi^{Y(D)}_{ij} d_i d_j\r) N^{\rm loc, \psi}_d(Y(D)) \,.
\eea 

\subsubsection{The topological vertex}
The connected generating functions $\cO_{\vec \mu}(Y^{\rm op}(D))(Q, \hbar)$ of open GW invariants in the `winding-number basis' of \cite{Li:2004uf},
\bea
\cO_{\vec \mu}(Y^{\rm op}(D))(Q, \hbar) &\coloneqq & \sum_\beta \mathsf{O}_{\beta,\vec \mu}(Y^{\rm op}(D))(\hbar) Q^\beta \,,
\label{eq:openfr2}
\eea
can be reconstructed from the disconnected generating functions $\cZ_{\vec\mu}(Y^{\rm op}(D))$ and $\cW_{\vec\mu}(Y^{\rm op}(D))$ in the `winding number' and `representation' bases defined by
\bea
\exp\l[
\sum_{\vec\mu \in (\cP)^{l-1}}\cO_{\vec \mu}(Y^{\rm op}(D))(Q, \hbar) \mathsf{x}_{\vec \mu}\r]  &=:& \sum_{\vec\mu \in (\cP)^{l-1}}  \cZ_{\vec\mu}(Y^{\rm op}(D))(Q,\hbar)  \vec{\mathsf{x}}_{\vec \mu}, \nn \\
&=:& \sum_{\vec\mu , \vec \nu \in (\cP)^{l-1}} \prod_{i=1}^{l-1} \frac{\chi_{\nu_i}(\mu_i)}{z_{\mu_i}} \cW_{\vec\nu}(Y^{\rm op}(D))(Q, \hbar)  \mathsf{x}_{\vec \mu}\,. \nn \\
\label{eq:opengf}
\eea
In the equation above, $\cP$ denotes the set of partitions and $\chi_{\a}(\b)$ denotes the value of the irreducible character of the symmetric group $S_{|\a|}$ on the conjugacy class labelled by the partition $\b$. \\

\begin{figure}[t]
\begingroup%
  \makeatletter%
  \providecommand\color[2][]{%
    \errmessage{(Inkscape) Color is used for the text in Inkscape, but the package 'color.sty' is not loaded}%
    \renewcommand\color[2][]{}%
  }%
  \providecommand\transparent[1]{%
    \errmessage{(Inkscape) Transparency is used (non-zero) for the text in Inkscape, but the package 'transparent.sty' is not loaded}%
    \renewcommand\transparent[1]{}%
  }%
  \providecommand\rotatebox[2]{#2}%
  \newcommand*\fsize{\dimexpr\f@size pt\relax}%
  \newcommand*\lineheight[1]{\fontsize{\fsize}{#1\fsize}\selectfont}%
  \ifx\svgwidth\undefined%
    \setlength{\unitlength}{255.54383045bp}%
    \ifx\svgscale\undefined%
      \relax%
    \else%
      \setlength{\unitlength}{\unitlength * \real{\svgscale}}%
    \fi%
  \else%
    \setlength{\unitlength}{\svgwidth}%
  \fi%
  \global\let\svgwidth\undefined%
  \global\let\svgscale\undefined%
  \makeatother%
  \begin{picture}(1,0.53017675)%
    \lineheight{1}%
    \setlength\tabcolsep{0pt}%
    \put(0,0){\includegraphics[width=\unitlength,page=1]{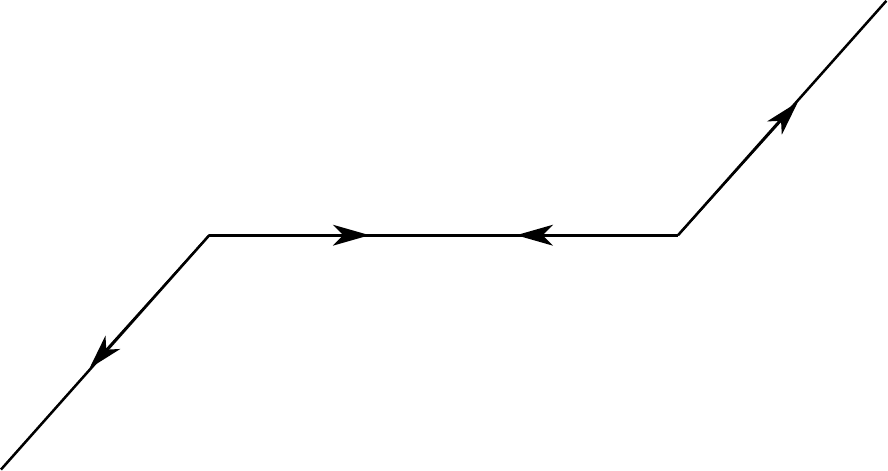}}%
    \put(0.27715295,0.45564018){\color[rgb]{0,0,0}\makebox(0,0)[lt]{\lineheight{1.25}\smash{\begin{tabular}[t]{l}$\mathsf{f}_1$\end{tabular}}}}%
    \put(0,0){\includegraphics[width=\unitlength,page=2]{toricgraph4.pdf}}%
    \put(0.91243334,0.36230617){\color[rgb]{0,0,0}\makebox(0,0)[lt]{\lineheight{1.25}\smash{\begin{tabular}[t]{l}$\mathsf{f}_2$\end{tabular}}}}%
    \put(0.66512096,0.06014346){\color[rgb]{0,0,0}\makebox(0,0)[lt]{\lineheight{1.25}\smash{\begin{tabular}[t]{l}$\mathsf{f}_3$\end{tabular}}}}%
    \put(0.09746216,0.24368211){\color[rgb]{0,0,0}\makebox(0,0)[lt]{\lineheight{1.25}\smash{\begin{tabular}[t]{l}$\mathsf{f}_4$\end{tabular}}}}%
  \end{picture}%
\endgroup%

\caption{The toric graph of the resolved conifold with four outer framed Lagrangians.}
\label{fig:conifold}
\end{figure}

The theory of the topological vertex consists of a glueing procedure to algorithmically compute $\cW_{\vec\nu}(Y^{\rm op}(D))$ from the representation-basis generating functions of toric Lagrangian triples $Y^{\rm op}(D)_i$, $i=1,2$ openly embedded into $Y(D)$. We content ourselves to state the consequences of the glueing algorithm in the two cases needed for for the study of Examples~I--III, referring the reader to  \cite[Sec.~6.1.2]{BBvG2} for a complete account. If $Y^{\rm op}(D)=\l(\mathrm{Tot}(\cO^{\oplus 2}_{\bbP^1}(-1)),L,\mathsf{f}\r)$ is the resolved conifold with $L$ the disconnected union of four outer branes in representations $\mu_1, \dots, \mu_4$ and framing shifts $\mathsf{f}=(f_1, \dots, f_4)$ as in \cref{fig:conifold}, we have
\bea
& &
\cW_{\mu_1,\mu_2,\mu_3,\mu_4}\big(\mathrm{Tot}\big(\cO^{\oplus 2}_{\bbP^1}(-1)\big),L, \mathsf{f}\big)(Q,\hbar) = \nn \\ &=& \sum_{\nu \in \cP} \cW_{\mu_1,\mu_2,\nu}\big(\bbC^3,(L_1, L_2, L_0),(f_1, f_2, f_0)\big)(\hbar) (-Q)^{|\nu|} \cW_{\nu^T,\mu_3,\mu_4}\big(\bbC^3,L,(-f_0, f_3,f_4)\big)(\hbar)\,, \nn \\
\label{eq:glueing}
\eea 
and if $Y^{\rm op}(D)=(\bbC^3,L,\mathsf{f})$ is the affine space with three outer branes in representations $\mu_1, \dots, \mu_3$ and framing shifts $\mathsf{f}=(f_1, \dots, f_3)$ as in \cref{fig:3vertex} (the framed 3-legged vertex), we have \cites{Aganagic:2003qj,moop}
\beq
\cW_{\mu_1,\mu_2,\mu_3}(\bbC^3,L,\mathsf{f})(\hbar) = q^{\kappa(\mu_1)/2} \prod_{i=1}^3 q^{f_i \kappa(\mu_i)/2}(-1)^{f_i|\mu_i|} \sum_{\delta \in \cP} s_{\frac{\mu_1^t}{\delta}}(q^{\rho+\mu_3})s_{\frac{\mu_2}{\delta}}(q^{\rho+\mu_3^t}) s_{\mu_3}(q^{\rho})\,.
\label{eq:vertex}
\eeq
where $s_\mu(q^{\rho+\a})$ denotes the principally-specialised $\a$-shifted Schur function in the representation of $\mathrm{GL}(\infty)$ labelled by $\mu$ (see \cite[Appendix~C]{BBvG2} for details),  $\kappa(\mu)$ is its second Casimir invariant of the partition $\mu$ normalised as $\kappa((1))=0$, and again $q=\re^{\ri \hbar}$.

\begin{figure}[t]
\begingroup%
  \makeatletter%
  \providecommand\color[2][]{%
    \errmessage{(Inkscape) Color is used for the text in Inkscape, but the package 'color.sty' is not loaded}%
    \renewcommand\color[2][]{}%
  }%
  \providecommand\transparent[1]{%
    \errmessage{(Inkscape) Transparency is used (non-zero) for the text in Inkscape, but the package 'transparent.sty' is not loaded}%
    \renewcommand\transparent[1]{}%
  }%
  \providecommand\rotatebox[2]{#2}%
  \newcommand*\fsize{\dimexpr\f@size pt\relax}%
  \newcommand*\lineheight[1]{\fontsize{\fsize}{#1\fsize}\selectfont}%
  \ifx\svgwidth\undefined%
    \setlength{\unitlength}{147.34940056bp}%
    \ifx\svgscale\undefined%
      \relax%
    \else%
      \setlength{\unitlength}{\unitlength * \real{\svgscale}}%
    \fi%
  \else%
    \setlength{\unitlength}{\svgwidth}%
  \fi%
  \global\let\svgwidth\undefined%
  \global\let\svgscale\undefined%
  \makeatother%
  \begin{picture}(1,1.01962887)%
    \lineheight{1}%
    \setlength\tabcolsep{0pt}%
    \put(0,0){\includegraphics[width=\unitlength,page=1]{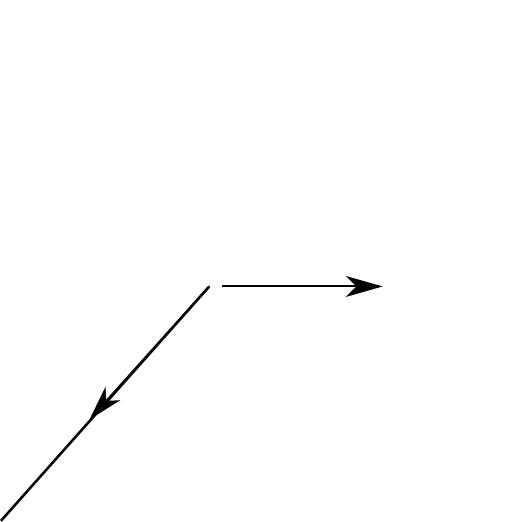}}%
    \put(0.14866089,0.37867717){\color[rgb]{0,0,0}\makebox(0,0)[lt]{\lineheight{1.25}\smash{\begin{tabular}[t]{l}$\mathsf{f}_3$\end{tabular}}}}%
    \put(0,0){\includegraphics[width=\unitlength,page=2]{3vertexfrac.pdf}}%
    \put(0.47548031,0.89870801){\color[rgb]{0,0,0}\makebox(0,0)[lt]{\lineheight{1.25}\smash{\begin{tabular}[t]{l}$\mathsf{f}_1$\end{tabular}}}}%
    \put(0.84625179,0.3215797){\color[rgb]{0,0,0}\makebox(0,0)[lt]{\lineheight{1.25}\smash{\begin{tabular}[t]{l}$\mathsf{f}_2$\end{tabular}}}}%
    \put(0,0){\includegraphics[width=\unitlength,page=3]{3vertexfrac.pdf}}%
  \end{picture}%
\endgroup%

\caption{The toric graph of the framed 3-legged vertex.}
\label{fig:3vertex}
\end{figure}





\section{Example I}\label{sec:I}

Let $a,b$ be coprime positive integers. Then $\bbP(1,a,b)$ has toric divisors $D_{(-1,0)}$, $D_{(0,-1)}$ and $D_{(b,a)}$ with relations
\beq
D_{(-1,0)} \sim b D_{(b,a)}\,, \quad D_{(0,-1)} \sim a D_{(b,a)}\,,
\eeq
intersections
\beq
D_{(-1,0)} \cdot D_{(0,-1)}=1\,, \quad D_{(0,-1)} \cdot D_{(b,a)}=\frac{1}{b}\,, \quad D_{(-1,0)} \cdot D_{(b,a)}=\frac{1}{a}\,,
\eeq
and self-intersections
\beq
D^2_{(-1,0)}=\frac{b}{a}\,, \quad D^2_{(0,-1)}=\frac{a}{b}\,, \quad D^2_{(b,a)}=\frac{1}{ab}\,.
\eeq 
To obtain a log smooth nef log Calabi--Yau surface, we choose $D_1=D_{(b,a)}$ and $D_2$ a smooth element of the linear system of $D_{(-1,0)}+D_{(0,-1)}$ so that
\beq
D_1^2=\frac{1}{ab}\,, \quad D_2^2=\frac{(a+b)^2}{ab}\,.
\eeq
Writing $D=D_1+D_2$, the topological Euler characteristic of the complement of $D$ is 
\beq \chi(\bbP(1,a,b)\setminus D)=1\,.
\eeq

Denote by $H$ the effective generator of $\hhh_2(\PP(1,a,b),\ZZ)$. Notice that $H^2=ab$, $D_1\sim \frac{1}{ab}H$ and $D_2\sim \frac{a+b}{ab}H$. Consequently, for a curve class $d=dH$, 
\beq
d\cdot D_1=d, \quad d \cdot D_2 = d(a+b)\,.
\eeq
\subsection{Local GW invariants}

Let $T\simeq(\bbC^\star)^2\circlearrowright E_{\exI}$ be the fibre-wise action on $E_{\exI}$, and denote by $\lambda_i$, $i=1,2$ its equivariant parameters. The inverse of the Gram matrix of the $T$-equivariant Poincar\'e pairing in the basis $\{1,H, H^2\}$ reads
\beq 
\eta^{-1}_{\exI} = 
\left(
\begin{array}{ccc}
 0 & 0 & \frac{\lambda _1 \lambda _2}{a b} \\
 0 & \frac{\lambda _1 \lambda _2}{a b} & -\frac{a \lambda _1+b \lambda _1+\lambda _2}{a^2 b^2} \\
 \frac{\lambda _1 \lambda _2}{a b} & -\frac{a \lambda _1+b \lambda _1+\lambda _2}{a^2 b^2} & \frac{a+b}{a^3 b^3} \\
\end{array}
\right) \,.
\label{eq:etaI}
\eeq
From \cite[Thm~0.1]{MR1653024} and \cite[Thm~4.6]{MR2510741}, the $T$-equivariant $I$-function of $E_{\exI}$ is
\beq 
I^{\exI}(y,z) = z y^{H/z}\sum_{d \geq 0} y^d \frac{\prod_{m=0}^{d-1}(\lambda_1 - H/(ab)-m z)\prod_{m=0}^{(a+b)d-1}(\lambda_2 - (a+b)H/(ab)-m z)}{\prod_{m=1}^{d}(H/(ab)+m z)\prod_{m=1}^{ad}(H/b+m z)\prod_{m=1}^{bd}(H/a+m z)}\,.
\label{eq:IfunI}
\eeq 
We have
\beq 
I^{\exI}(y,z) = z+(\log y) H + \cO\l(\frac{1}{z}\r) \,,
\eeq 
and therefore, as alluded to in \cref{sec:localtech}, the mirror map is trivial,
\beq 
J_{\rm small}^{\exI}(t,z)=I^{\exI}(e^t,z)\,.
\eeq 
Since $l=2$ and thus $N_d^{\rm loc,\psi}(\exI)=N_d^{\rm loc}(\exI)$,  \eqref{eq:NlocpsiJ} gives
\beq 
N_d^{\rm loc}(\exI)= \frac{1}{a b} \sum_{i \in \{0,1,2\}} (\eta_{\exI})_{2,i} [H^i z^{-1}y^d] I^{\exI}(y,z)\,,
\label{eq:NdlocI0}
\eeq 
where, from \eqref{eq:IfunI}, we have  \beq 
[z^{-1} y^d] I^{\exI}(y,z) = \frac{(-1)^{d (a+b+1)} \Gamma ((a+b) d) (H-a b \lambda_1 ) (H (a+b)-a b \lambda_2 )}{a^2 b^2 d \Gamma (a d+1) \Gamma (b d+1)} \,,
\label{eq:IfunIzm1}
\eeq 
and we have used that $H^2=(ab) \mathrm{pt}$. Combining \eqref{eq:etaI}, \eqref{eq:NdlocI0} and \eqref{eq:IfunIzm1} finally yields
\beq
N_d^{\rm loc}(\exI)=\frac{(-1)^{d (a+b+1)}}{d^2 (a+b)} \binom{(a+b) d}{a d}\,.
\label{eq:NdlocI}
\eeq 
\subsection{Log GW invariants}

\begin{prop} \label{prop:P1ab}
Let $a,b$ be coprime positive integers. Then
\[
\mathsf{N}^{\rm log}_d(\mathbb{P}(1,a,b))(\hbar)=\qbinom{(a+b)d}{ad}_q.
\label{eq:NdlogI}
\]
\end{prop}

In the $\hbar\to 0$ ($q\to 1$) limit, this recovers the log-local correspondence of \cref{thm:log-local} for $Y(D)=\exI$.

\begin{figure}[t]
\begin{tikzpicture}[smooth, scale=1.2]
\draw[<->] (-2.5,2) to (-1.5,2);
\draw[<->] (-2,1.5) to (-2,2.5);
\node at (-1.67,2.15) {$\scriptstyle{x}$};
\node at (-2.15,2.3) {$\scriptstyle{y}$};
\draw[step=1cm,gray,very thin] (-2.5,-2.5) grid (2.5,2.5);
\draw[thick] (2.5,0) to (-2.5,0);
\draw[thick] (0,-2.5) to (0,0);
\draw[thick] (0,0) to (1.25,2.5);
\node at (-0.4,-1.7) {$\overline{D_2}$};
\node at (0.9,1.2) {$\overline{D_1}$};
\node at (-1.5,0.3) {$\overline{D_3}$};
\node at (-2.3,0) {$\times$};
\node at (0.55,2.2) {$\sstyle{(b,a+b)}$};
\node at (1,2) {$\sstyle{\bullet}$};
\node at (-0.5,0.2) {$\sstyle{1+tx^{-1}}$};
\end{tikzpicture}
\caption{The toric model of $\exI$.}
\label{fig:P1ab}
\end{figure}

\begin{proof}
We follow the computational technique sketched out in Section \ref{sec:scattdiags}, and given in full detail in \cite[Section 4.2]{BBvG2}. We start by finding a toric model for $(\bbP(1,a,b),D_1+D_2)$.
We view $\bbP(1,a,b)$ as given by the fan generated by $(-1,0)$, $(0,-1)$ and $(b,a)$. We add a ray in the direction $(-1,1)$ which yields a new divisor $D_3$ and a birational map
\[
\varphi : \l(\widetilde{\bbP(1,a,b)},(D_1-D_3)+(D_2-D_3)+D_3\r) \to \l(\bbP(1,a,b),D_1+D_2\r).
\]
Then the proper transform of $D_{(-1,0)}$ is a $(-1)$-curve, which we contract:
\[
\pi : \l(\widetilde{\bbP(1,a,b)},(D_1-D_3)+(D_2-D_3)+D_3\r) \to \l(\overline{\bbP(1,a,b)},\overline{D_1}+\overline{D_2}+\overline{D_3}\r).
\]
The complement of the proper transform of $\overline{D}=\overline{D_1}+\overline{D_2}+\overline{D_3}$ now has Euler characteristic 0, hence is $(\bbC^*)^2$, therefore the variety is toric.
We apply the $\mathrm{SL}_2(\bbZ)$ transformation given by
$ \begin{pmatrix}
1 & 0 \\
1 & 1 \\
\end{pmatrix}$
to the fan and obtain the toric model of \cref{fig:P1ab}.
The toric surface $\overline{\bbP(1,a,b)}$ is the weighted projective space $\bbP(1,a+b,b)$. $\pi$ is determined by blowing up a smooth point on the divisor $\overline{D_3}$.
This yields a focus-focus singularity on the ray directed by $(-1,0)$ corresponding to $\overline{D_3}$.

Next we multiply the theta functions corresponding to $D_1$ and $D_2$ and extract the identity component. To do so, we consider all the ways of combining two broken lines, one coming from the direction $D_1$ with weight $d\cdot D_1$, one coming from the direction $D_2$ with weight $d\cdot D_2$, at a fixed point $p$, so that they are opposite at $p$. Then the desired log Gromov--Witten invariant is the sum of the product over the end-coefficients of all possible ways of combining two broken lines in that way.
There is only one configuration that solves this problem. This starts with a broken line $\beta_1$ coming from the $D_1$-direction carrying the monomial $(x^by^{a+b})^{d\cdot D_1}=x^{bd}y^{(a+b)d}$ and a broken line $\beta_2$ coming from the $D_2$-direction carrying the monomial $(y^{-1})^{d\cdot D_2}=y^{-(a+b)d}$. The focus-focus singularity produces a wall with wall-crossing function $1+tx^{-1}$, see \cite[Section 4.2]{BBvG2}. When $\beta_2$ crosses this wall, it picks up a contribution from the wall-crossing function, see \cite[Section 4.2]{BBvG2}. In order to be opposite to $\beta_1$ at $p$, $\beta_2$ picks up the wall-crossing contribution that has degree $-bd$ in $x$. In fact, the resulting product of broken lines needs to respect the intersection profile of $d$. Since the focus-focus singularity corresponds to $D_{(-1,0)}$
and $d\cdot D_{(-1,0)}=bd$, the product of the end-coefficients of $\beta_1$ and $\beta_2$ needs to have degree $bd$ in $t$, see \cite[Section 4.2]{BBvG2}. Either condition determines the wall-crossing as sketched in Figure \ref{fig:P1abscatt} and $\beta_2$ picks up the $q$-binomial coefficient $\qbinom{(a+b)d}{bd}_q=\qbinom{(a+b)d}{ad}_q$ which is $\mathsf{N}^{\rm log}_d(\mathbb{P}(1,a,b))(q)$ by \cite[Proposition 4.2]{BBvG2}. 
\end{proof}

\begin{figure}[t]
\begin{tikzpicture}[smooth, scale=1.2]
\draw[<->] (-2.5,2) to (-1.5,2);
\draw[<->] (-2,1.5) to (-2,2.5);
\node at (-1.67,2.15) {$\scriptstyle{x}$};
\node at (-2.15,2.3) {$\scriptstyle{y}$};
\draw[step=1cm,gray,very thin] (-2.5,-2.5) grid (3.5,2.5);
\draw[thick] (3.5,0) to (-2.5,0);
\draw[thick] (0,-2.5) to (0,0);
\draw[thick] (0,0) to (1.25,2.5);
\node at (-0.4,-1.7) {$\overline{D_2}$};
\node at (0.9,1.2) {$\overline{D_1}$};
\node at (-1.5,0.3) {$\overline{D_3}$};
\node at (-2.3,0) {$\times$};
\node at (0.55,2.2) {$\sstyle{(b,a+b)}$};
\node at (-0.5,0.2) {$\sstyle{1+tx^{-1}}$};
\draw[<->,thick] (1.5,-2.5) to (1.5,0) to (2.75,2.5);
\node at (2.1,-2.2) {$\sstyle{y^{-(a+b)d}}$};
\node at (3.3,2.2) {$\sstyle{x^{bd}y^{(a+b)d}}$};
\node at (3.2,0.4) {$\sstyle{\qbinom{(a+b)d}{bd}_qt^{bd}x^{-bd}y^{-(a+b)d}}$};
\node at (2,1) {$\bullet$};
\node at (2.2,1.1) {$\sstyle{p}$};
\end{tikzpicture}
\caption{$\scatt\exI$}
\label{fig:P1abscatt}
\end{figure}

\subsection{Open GW invariants}
By \cite[Construction~6.4]{BBvG2}, $(\exI)^{\rm op}$ is a toric Lagrangian triple $(X,L,\mathsf{f})$ depicted by the toric graph in \cref{fig:1vertex}. In particular, deleting $D_1=D_{(b,a)}$ in $\bbP(1,a,b)$ gives two dimensional affine space $\bbC^2$, and therefore $X=K_{\bbP(1,a,b)\setminus D_{(b,a)}} = K_{\bbC^2}\simeq \bbC^3$. The compactification to $K_{\bbP(1,a,b)}$ induces a framing on an outer Lagrangian $L$ on $\bbC^3$ shifted by $b/a$ with respect to the canonical framing, and the finite index morphism $\iota: \hhh_2^{\rm rel}(\bbC^3,L,\bbZ) \simeq \hhh_1(L, \bbZ) \to A_1(\bbP(1,a,b))$ is defined by $[S^1]\to [D_{(-1,0)}] = H/a$. In particular, the winding number $ad$ around $L$ corresponds to the curve class $dH$. 

The all-genus, 1-hole open GW generating function of $(\exI)^{\rm op}$ at winding $j=a d$ is, by \eqref{eq:openfr2} and \eqref{eq:vertex}:
\bea
\mathsf{O}_{j}((\exI)^{\rm op}) &=& \sum_{s=0}^{j-1}\frac{(-1)^{s}}{j} \cW_{(j-s,1^s),\emptyset,\emptyset}(\bbC^3,L,\mathsf{f}) \nn \\ &=& \sum_{s=0}^{j-1} 
\frac{(-1)^{s}}{j} q^{b/a\kappa((j-s,1^s)/2)}(-1)^{b/a j} s_{(j-s,1^s)}(q^\rho) \,,
\eea 
where $(j-s,1^s)$ denotes a partition represented by a hook Young diagram with $j$ boxes and $s+1$ rows. Using the hook formula for Schur functions \cite{MR325407},
\beq 
s_{(j-s,1^s)}(q^\rho) = \frac{q^{\frac{1}{2}\binom{j}{2}-\frac{j s}{2}}}{[j]_q [j-s]_q! [s]_q!}\,,
\eeq 
we get
\bea
\mathsf{O}_{j}((\exI)^{\rm op}) &=& 
  \frac{(-1)^{b/aj} q^{\l(\frac{b}{a}+\frac{1}{2}\r)\binom{j }{2}}}{j [j]_q} \sum_{s=0}^{d-1} \frac{(-1)^{s}  q^{-\l(\frac{b}{a}+\frac{1}{2}\r)j s}}{[j-s-1]_q![s]_q!}, \nn \\
 &=&  \frac{(-1)^{b/aj} q^{\l(\frac{b}{a}+\frac{1}{2}\r)\binom{j}{2}}}{j [j]_q!} \sum_{s=0}^{j-1}   \qbinom{j-1}{s}_q \big(-q^{bj/a}\big)^s   q^{-\frac{j s}{2}},  \nn \\
 &=&  \frac{(-1)^{b/aj} q^{\l(\frac{b}{a}+\frac{1}{2}\r)\binom{j}{2}}}{j [j]_q!} \prod_{k=1}^{j-1}   \big(1-q^{-k-bj/a}\big), \nn \\
  &=&  \frac{(-1)^{b/aj}}{j [j]_q!}  \frac{[(1+b/a)j-1]_q!}{[b/a j]_q!} = \frac{(-1)^{jb/a}}{j [(1+b/a)j]_q}  \qbinom{(1+b/a)j}{j}_q\,.
\eea 
In this case, the map $\iota: \hhh_1(L,\bbZ) \to {\rm A}_1(\bbP(1,a,b))$ of \eqref{eq:iota} is 
\beq
\iota : [S^1] \mapsto [D_{(-1,0)}]=\frac{H}{a},
\eeq 
therefore, setting $j=da$,
\beq
\mathsf{O}_{d a}((\exI)^{\rm op})=\frac{(-1)^{d b}}{d a [(a+b) d]_q} \qbinom{(a+b) d}{a d}_q\,,
\eeq 
which proves \cref{thm:logopen} for $\exI$.

\begin{figure}[t]
\begingroup%
  \makeatletter%
  \providecommand\color[2][]{%
    \errmessage{(Inkscape) Color is used for the text in Inkscape, but the package 'color.sty' is not loaded}%
    \renewcommand\color[2][]{}%
  }%
  \providecommand\transparent[1]{%
    \errmessage{(Inkscape) Transparency is used (non-zero) for the text in Inkscape, but the package 'transparent.sty' is not loaded}%
    \renewcommand\transparent[1]{}%
  }%
  \providecommand\rotatebox[2]{#2}%
  \newcommand*\fsize{\dimexpr\f@size pt\relax}%
  \newcommand*\lineheight[1]{\fontsize{\fsize}{#1\fsize}\selectfont}%
  \ifx\svgwidth\undefined%
    \setlength{\unitlength}{175.84818442bp}%
    \ifx\svgscale\undefined%
      \relax%
    \else%
      \setlength{\unitlength}{\unitlength * \real{\svgscale}}%
    \fi%
  \else%
    \setlength{\unitlength}{\svgwidth}%
  \fi%
  \global\let\svgwidth\undefined%
  \global\let\svgscale\undefined%
  \makeatother%
  \begin{picture}(1,0.87990461)%
    \lineheight{1}%
    \setlength\tabcolsep{0pt}%
    \put(0,0){\includegraphics[width=\unitlength,page=1]{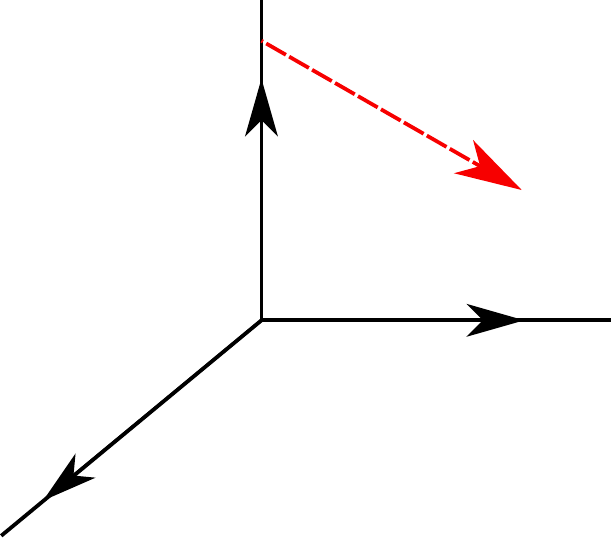}}%
    \put(0.54526264,0.65736578){\makebox(0,0)[lt]{\lineheight{1.25}\smash{\begin{tabular}[t]{l}$\mathsf{f}$\end{tabular}}}}%
  \end{picture}%
\endgroup%

\caption{The toric graph of $(\exI)^{\rm op}$, depicted for $a=3$, $b=2$.}
\label{fig:1vertex}
\end{figure}

\section{Example II}
\label{sec:II}

Define $\exII$ by considering the blow up
$\pi : 
\exII \longrightarrow \exI 
$
at a smooth point of the second divisor of $\bbP(1,a,b)$. Torically, we take the fan with rays $(-1,0)$, $(-1,-1)$ $(0,-1)$ and $(b,a)$ and choose $D_1=D_{(b,a)}$ as well as $D_2$ a smooth member of the linear system $D_{(-1,0)}+D_{(-1,-1)}+D_{(0,-1)}$. The relations are
\beq
D_{(-1,0)} + D_{(-1,-1)} \sim b D_1, \quad D_{(0,-1)} + D_{(-1,-1)} \sim a D_1,
\eeq
the intersections are
\beq
D_{(-1,0)} \cdot D_{(-1,-1)}=D_{(-1,-1)} \cdot D_{(0,-1)}=1, \quad D_{(0,-1)} \cdot D_{(b,a)}=\frac{1}{b}, \quad D_{(-1,0)} \cdot D_{(b,a)}=\frac{1}{a},
\eeq
and the self-intersections are
\beq
D^2_{(-1,0)}=\frac{b}{a}-1, \quad D^2_{(-1,-1)}=-1, \quad D^2_{(0,-1)}=\frac{a}{b}-1, \quad D^2_{(b,a)}=\frac{1}{ab}.
\eeq 
Writing $D=D_1+D_2$, the topological Euler characteristic of the complement of $D$ is $\chi(Y^{[2]}_{(a,b)}\setminus D) = 2$. Then, $\hhh_2(Y_{(a,b)},\bbZ)$ is generated by the proper transform $\pi^*H\sim ab D_1$ and by the class of $D_{(-1,-1)}$. We write an effective curve class as $d=d_0 f+d_1 E$ with
\beq
E \coloneqq D_{(-1,-1)}, \quad f \coloneqq \pi^*H-D_{(-1,-1)} \,,
\label{eq:fE}
\eeq 
so that
\beq 
d\cdot D_1 = d_0, \quad d \cdot D_2 = (a+b-1)d_0+d_1, \quad d\cdot \pi^*D_{(-1,0)} = bd_0 \quad d\cdot D_{(-1,-1)}=d_0-d_1\,.
\eeq




\subsection{Local GW invariants}

In this case, the (inverse) Gram matrix of the $T$-equivariant Poincar\'e pairing in the basis $\{\varphi_0, \varphi_1, \varphi_2, \varphi_3\}=\{1, E, f, \mathrm{pt}\}$ reads
\beq 
\eta^{-1}_{{\exII}} =
\left(
\begin{array}{cccc}
 0 & 0 & 0 & \lambda _1 \lambda _2 \\
 0 & -\frac{\lambda _1 \lambda _2 (a b-1)}{a b} & \frac{\lambda _1 \lambda _2}{a b} & -\frac{-a (b-1) \lambda _1+b \lambda
   _1+\lambda _2}{a b} \\
 0 & \frac{\lambda _1 \lambda _2}{a b} & \frac{\lambda _1 \lambda _2}{a b} & -\frac{a \lambda _1+b \lambda _1+\lambda _2}{a b} \\
 \lambda _1 \lambda _2 & -\frac{-a (b-1) \lambda _1+b \lambda _1+\lambda _2}{a b} & -\frac{a \lambda _1+b \lambda _1+\lambda
   _2}{a b} & \frac{1}{a}+\frac{1}{b} \\
\end{array}
\right)\,,
\label{eq:etaII}
\eeq 
and the $T$-equivariant $I$-function is
\bea
I^{\exII}(y,z) &=& 
  \sum_{d_0, d_1\geq 0} \Bigg[ \frac{z y_0^{p_0/z+p_1/z}y_0^{d_0}y_1^{d_1}\prod_{m=0}^{d_0-1}(\lambda_1-p_0-m z) }{\prod_{m=1}^{d_0}(p_0+m z)\prod_{m=1}^{(b-1)d_0+d_1}((b-1)p_0+p_1+m z)}
\nn \\ & & \frac{ \prod_{m=0}^{(a+b-1)d_0+d_1-1}(\lambda_2+(1-b-a) p_0-p_1-m z)}{\prod_{m=1}^{d_0-d_1}(p_0-p_1+m z)\prod_{m=1}^{(a-1)d_0+d_1}((a-1)p_0+p_1+m z)}\Bigg]\,,
\label{eq:IfunII}
\eea  
where $p_0 \coloneqq H/(ab)$, $p_1 \coloneqq H/(ab)-E$. At $\cO(1)$ around $z=\infty$ we have
\beq 
I^{\exII}(y,z) = z+(\log y_0) p_0+ (\log y_1) p_1 + \cO\l(\frac{1}{z}\r) \,,
\eeq 
so the mirror map is once again trivial:
\beq 
J_{\rm small}^{\exII}(t,z)=I^{\exII}(e^t,z)\,.
\eeq 
From \eqref{eq:NlocpsiJ} we obtain
\beq 
N_d^{\rm loc}(\exII)= \sum_{i \in \{0,1,2,3\}} \big(\eta_{\exII}\big)_{3,i} [\varphi_i ~z^{-1}y^d] I_{\exII}(y,z)\,,
\label{eq:NdlocII}
\eeq 
where, from \eqref{eq:IfunII}, we have  
\beq 
[z^{-1}y^d] I^{{\exII}}(y,z) = \frac{\Gamma \left(d_0\right) \left(p_0-\lambda _1\right) (-1)^{d_0 (a+b)+d_1} \Gamma \left(a d_0+b d_0-d_0+d_1\right) \left(p_0
   (a+b-1)-\lambda _2+p_1\right)}{\Gamma \left(d_0+1\right) \Gamma \left(d_0-d_1+1\right) \Gamma \left(a d_0-d_0+d_1+1\right)
   \Gamma \left(b d_0-d_0+d_1+1\right)}\,.
\label{eq:IfunIIzm1}
\eeq 
Piecing \eqref{eq:etaII}, \eqref{eq:NdlocII} and \eqref{eq:IfunIIzm1} together finally leads to
\beq
N_d^{\rm loc}\l(\exII\r)=\frac{(-1)^{d_0 (a+b)+d_1} \left((a+b-1) d_0+d_1-1\right)!}{d_0 \left(d_0-d_1\right)!
   \left((a-1) d_0+d_1\right)! \left((b-1) d_0+d_1\right)!}\,.
   \label{eq:NdlocIIfin}
\eeq 

\subsection{Log GW invariants}

\begin{prop}
Let $a,b$ be coprime positive integers. Then
\beq
\mathsf{N}^{\rm log}_d\left(\exII\right)(q) = \qbinom{ad_0}{d_0-d_1}_q \qbinom{(a+b-1)d_0+d_1}{ad_0}_q.
\label{eq:NdlogII}
\eeq
\end{prop}

 In conjunction with \eqref{eq:NdlocIIfin}, taking the $q \to 1$ limit recovers the numerical log-local correspondence of  \cref{thm:log-local}.

\begin{proof}
Given that $Y_{(a,b)}$ differs from $\bbP(1,a,b)$ by blowing up a smooth point of $D_2$, the toric model of $Y^{[2]}_{(a,b)}$ differs from that of $\bbP(1,a,b)$ by additionally blowing up a smooth point of $\overline{D_2}$. At the level of scattering diagrams, we start with the one for $\exI$ as in Figure \ref{fig:P1abscatt} and add a focus-focus singularity in the direction $(0,-1)$ (corresponding to $\overline{D_2}$). This creates simple scattering (see \cite[Section 4.2]{BBvG2}) as described in Figure \ref{fig:Yabscatt}. The broken line calculation proceeds as for $\exI$, with now both broken lines crossing a wall. Which walls which broken lines cross depends on the location of $p$. However, the end-result of multiplying the theta functions is independent from the location of $p$, see \cite{GHS16}. Figure \ref{fig:Yabscatt} shows where the broken lines bend. As for $\exI$, the bending is determined by the intersection profile of $d$ with $D_1$, $D_2$ and the two exceptional divisors corresponding to the focus-focus singularities.

\begin{figure}[h]
\begin{tikzpicture}[smooth, scale=1.2]
\draw[<->] (-2.5,2) to (-1.5,2);
\draw[<->] (-2,1.5) to (-2,2.5);
\node at (-1.67,2.15) {$\scriptstyle{x}$};
\node at (-2.15,2.3) {$\scriptstyle{y}$};
\draw[step=1cm,gray,very thin] (-2.5,-2.5) grid (3.75,3.5);
\draw[thick] (3.75,0) to (-2.5,0);
\draw[thick] (-1,-2.5) to (-1,0);
\draw[thick] (-1,0) to (0.25,2.5);
\draw[thick] (3,-2.5) to (3,3.5);
\node at (-1.4,-1.7) {$\overline{D_2}$};
\node at (-0.1,1.2) {$\overline{D_1}$};
\node at (-2.3,0) {$\times$};
\node at (3,-2.3) {$\times$};
\node at (-0.45,2.2) {$\sstyle{(b,a+b)}$};
\node at (-1.5,0.2) {$\sstyle{1+t_0x^{-1}}$};
\node[rotate=90] at (2.8,-1.5) {$\sstyle{1+t_1y^{-1}}$};
\node[rotate=45] at (3.7,0.9) {$\sstyle{1+t_0t_1x^{-1}y^{-1}}$};
\draw[thick] (3,0) to (4.5,1.5);
\draw[<->,thick] (0,-2.5) to (0,0) to (3,2) to (3.75,3.5);
\node at (1,-1.6) {$\sstyle{y^{-(a+b-1)d_0-d_1}}$};
\node at (4.15,2.7) {$\sstyle{x^{bd_0}y^{(a+b)d_0}}$};
\node at (0.7,0.2) {$\circleed{2}$};
\node at (2.8,2.2) {$\circleed{1}$}; 
\node at (1.5,1) {$\bullet$};
\node at (1.6,0.9) {$\sstyle{p}$};
\end{tikzpicture}
\caption{$\scatt \exII$}
\label{fig:Yabscatt}
\end{figure}

At $\circleed{1}$, the broken line coming from the $D_1$-direction picks up
\[
\qbinom{bd_0}{d_0-d_1}_qt_1^{d_0-d_1}x^{bd_0}y^{(a+b-1)d_0+d_1}.
\]
At $\circleed{2}$, the broken line coming from the $D_2$-direction picks up 
\[
\qbinom{(a+b-1)d_0+d_1}{bd_0}_qt_0^{bd_0}x^{-bd_0}y^{-(a+b-1)d_0-d_1}.
\]
We conclude by \cite[Proposition 4.2]{BBvG2}, noting that
\[
\qbinom{bd_0}{d_0-d_1}_q \qbinom{(a+b-1)d_0+d_1}{bd_0}_q = \qbinom{ad_0}{d_0-d_1}_q \qbinom{(a+b-1)d_0+d_1}{ad_0}_q.
\]
\end{proof}

\subsection{Open GW invariants}

To construct $(\exII)^{\rm op}$, as for $\exI$, we take the canonical bundle on the complement of $D_1=D_{(b,a)}$. From the toric description of $\exII$ we have that $Y^{[2]}_{(b,a)} \setminus D_{(b,a)}$ is described by a fan with 1-dimensional cones $D_{(-1,0)}$, $D_{(0,-1)}$ and $D_{(-1,-1)}$, hence $Y^{[2]}_{(b,a)} \setminus D_{(b,a)} \simeq \cO_{\bbP^1}(-1)$, and $K_{Y^{[2]}_{(b,a)} \setminus D_{(b,a)}} \simeq \cO_{\bbP^1}(-1)\oplus \cO_{\bbP^1}(-1)$. By \cite[Construction~6.4]{BBvG2}, the compactification to $K_{Y^{[2]}_{(a,b)}}$ induces a framing of $b/a-1$ on the outer Lagrangian corresponding to the deletion of $D_{(b,a)}$; see \cref{fig:conifold2}.

\begin{figure}[t]
\begingroup%
  \makeatletter%
  \providecommand\color[2][]{%
    \errmessage{(Inkscape) Color is used for the text in Inkscape, but the package 'color.sty' is not loaded}%
    \renewcommand\color[2][]{}%
  }%
  \providecommand\transparent[1]{%
    \errmessage{(Inkscape) Transparency is used (non-zero) for the text in Inkscape, but the package 'transparent.sty' is not loaded}%
    \renewcommand\transparent[1]{}%
  }%
  \providecommand\rotatebox[2]{#2}%
  \newcommand*\fsize{\dimexpr\f@size pt\relax}%
  \newcommand*\lineheight[1]{\fontsize{\fsize}{#1\fsize}\selectfont}%
  \ifx\svgwidth\undefined%
    \setlength{\unitlength}{278.98099884bp}%
    \ifx\svgscale\undefined%
      \relax%
    \else%
      \setlength{\unitlength}{\unitlength * \real{\svgscale}}%
    \fi%
  \else%
    \setlength{\unitlength}{\svgwidth}%
  \fi%
  \global\let\svgwidth\undefined%
  \global\let\svgscale\undefined%
  \makeatother%
  \begin{picture}(1,0.49957593)%
    \lineheight{1}%
    \setlength\tabcolsep{0pt}%
    \put(0,0){\includegraphics[width=\unitlength,page=1]{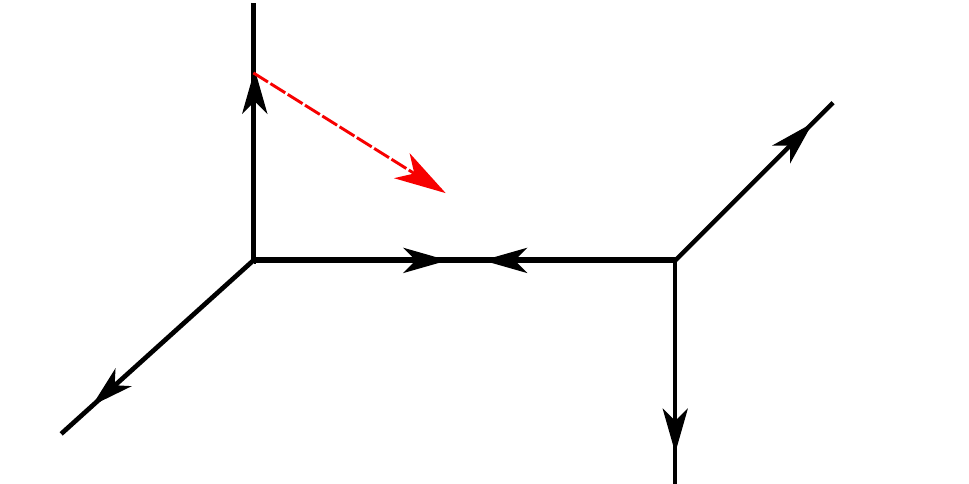}}%
    \put(0.30395601,0.40479229){\makebox(0,0)[lt]{\lineheight{1.25}\smash{\begin{tabular}[t]{l}$\mathsf{f}$\end{tabular}}}}%
  \end{picture}%
\endgroup%

\caption{The toric graph of $(\exII)^{\rm op}$, depicted for $a=3$, $b=5$.}
\label{fig:conifold2}
\end{figure}

From \eqref{eq:glueing} and \eqref{eq:vertex}, we have
\bea
\cW_\a((\exII)^{\rm op})(Q,\hbar) &=& q^{\l(\frac{b}{a}-1\r)\kappa(\a)/2}(-1)^{\l(\frac{b}{a}-1\r)|\a|}\sum_{\nu\in \cP}  s_{\nu^t}(q^{\rho+\a}) s_\a(q^\rho) s_\nu(-Q q^\rho) \nn \\ &=&
q^{\l(\frac{b}{a}-1\r)\kappa(\a)/2}(-1)^{\l(\frac{b}{a}-1\r)|\a|} s_\a(q^\rho)
\prod_{i,j\geq 1} \big(1-Q q^{-i-j+1+\a_i}\big)\,,
\eea
so that
\bea
\frac{\cW_{(j-s-1,1^s)}((\exII)^{\rm op})(Q,\hbar)}{\cW_{\emptyset}((\exII)^{\rm op})(Q,\hbar)}
&=&  \frac{(-1)^{(b/a-1) j}q^{\l(\frac{b}{a}-\frac{1}{2}\r)\l(\binom{j}{2} - j s\r)} \prod_{k=0}^{j-1} (1-q^{k} Q q^{-s})}{[j]_q [j-s-1]_q! [s]_q!}.
\eea
Using the Cauchy binomial theorem we get
\bea
& & 
\cO_j((\exII)^{\rm op})(Q,\hbar)  =  \sum_{s=0}^{j-1} \frac{(-1)^s}{j}
\frac{\cW_{(j-s-1,1^s)}((\exII)^{\rm op})(Q,\hbar)}{\cW_{\emptyset}((\exII)^{\rm op})(Q,\hbar)}
\nn \\
&=& \frac{(-1)^{(b/a-1) j}q^{\l(\frac{b}{a}-\frac{1}{2}\r)\binom{j}{2}}}{j [j]_q!} \sum_{l=0}^\infty q^{\frac{l(j+1)}{2}} \qbinom{j}{l}_q (-Q)^l q^{-l} \sum_{s=0}^{j-1} \qbinom{j-1}{s}_q (-q^{-(b/a-1)j-l})^{s}  q^{-\frac{1}{2}j s}  , \nn \\
&=& \frac{(-1)^{(b/a-1) j}q^{\l(\frac{b}{a}-\frac{1}{2}\r)\binom{j}{2}}}{j [j]_q!} \sum_{l=0}^\infty  q^{\frac{l(j-1)}{2}} \qbinom{j}{l}_q (-Q)^l \prod_{k=1}^{j-1} (1-q^{-(b/a-1)j-l-k}),
\eea 
so the $\cO(Q^l)$ coefficient reads
\beq
\mathsf{O}_{l;j}((\exII)^{\rm op})(\hbar)  = \frac{(-1)^{(b/a+1) j+l}}{j [b/aj+l]_q} \qbinom{b/a j+l}{j}_q \qbinom{j}{l}_q.
\eeq
The morphism  $\iota: \hhh_2(\bbP^1,\bbZ) \oplus \hhh_1(S^1,\bbZ) \to {\rm A}_1\l(\exII\r)$ in \eqref{eq:iota} reads
\bea
\iota : [S^1] & \mapsto & [D_{(-1,0)}] = \frac{H}{a} - E, \nn \\
\iota : [\bbP^1] & \mapsto & [D_{(-1,-1)}] = E \,,
\eea 
and, accordingly, the change-of-variables relating the curve degrees $(d_0, d_1)$ in $\hhh_2(Y_{(a,b)}, \bbZ)$ and the relative homology variables $(l; j)$ in $\hhh_2^{\rm rel}((Y_{(a,b)})^{\rm op},L,\bbZ)$ is 
\beq
\begin{array}{c}
 j\to d_0 a\,, \\
 l\to (a-1) d_0+d_1\,, \\
\end{array}
\eeq
verifying \cref{thm:logopen} for $\exII$.

\section{Example III}
\label{sec:III}

Start with the toric pair $\bbP(1,a,b)$ and blow up a smooth point of the toric divisor $D_{(-1,0)}\sim H/a$, see Figure \ref{fan:P1ab}. The result is a toric surface $Y^{[3]}_{(a,b)}$ with non-toric boundary $D=D_1+D_2+D_3$ determined as follows.
Torically, the fan of $Y^{[3]}_{(a,b)}$ is given by the rays directed by $(b,a)$, $(-1,0)$, $(-1,-1)$ and $(0,-1)$. The boundary $D$ consists of two toric divisors, $D_1$ corresponding to the ray $(b,a)$ and $D_2$ corresponding to the ray $(-1,0)$, as well as a non-toric divisor $D_3$ which is a smooth member of the linear system determined by the sum of the rays $(-1,-1)$ and $(0,-1)$. The topological Euler characteristic of the complement of $D$ is
$\chi(Y^{[3]}_{(a,b)}\setminus D)=1$. As before, we write an effective curve class as $d=d_0 f + d_1 E$, where $f,E \in \hhh_2(Y^{[3]}_{(a,b)},\bbZ)$ are defined in \eqref{eq:fE} and
\beq
d\cdot D_1 = d_0, \quad d\cdot D_2 = (b-1) d_0 + d_1, \quad d\cdot D_3 = a d_0, \quad d \cdot E = d_0-d_1.
\eeq

\subsection{Local GW invariants}

In this case, the (inverse) Gram matrix of the $T$-equivariant Poincar\'e pairing in the basis $\{\varphi_0, \varphi_1, \varphi_2, \varphi_3\}=\{1, E, f, \mathrm{pt}\}$ reads
\beq 
\eta^{-1}_{{\exIII}} =
\left(
\begin{array}{cccc}
 0 & 0 & 0 & \lambda _1 \lambda _2 \lambda _3 \\
 0 & -\frac{\lambda _1 \lambda _2 \lambda _3 (a b-1)}{a b} & \frac{\lambda _1 \lambda _2 \lambda _3}{a b} & \frac{b \lambda _1 \left(a \lambda _3-\lambda _2\right)-\lambda _3 \left(a \lambda _1+\lambda _2\right)}{a b} \\
 0 & \frac{\lambda _1 \lambda _2 \lambda _3}{a b} & \frac{\lambda _1 \lambda _2 \lambda _3}{a b} & -\frac{\lambda _3 \left(a \lambda _1+\lambda _2\right)+b \lambda _1 \lambda _2}{a b} \\
 \lambda _1 \lambda _2 \lambda _3 & \frac{b \lambda _1 \left(a \lambda _3-\lambda _2\right)-\lambda _3 \left(a \lambda _1+\lambda _2\right)}{a b} & -\frac{\lambda _3 \left(a \lambda _1+\lambda _2\right)+b \lambda _1 \lambda _2}{a b} & \frac{\lambda
   _2}{a}+\frac{\lambda _3}{b}+\lambda _1 \\
\end{array}
\right),
\eeq 
and the $T$-equivariant $I$-function is
\bea
I^{\exIII}(y,z) &=&
 \sum_{d_0, d_1\geq 0} \Bigg[ \frac{\prod_{m=0}^{d_0-1}(\lambda_1-p_0-m z) \prod_{m=0}^{(a-1)d_0+d_1-1}(\lambda_2+(1-a) p_0-p_1-m z)}{\prod_{m=1}^{d_0}(p_0+m z)\prod_{m=1}^{(a-1)d_0+d_1}((a-1)p_0+p_1+m z)}
\nn \\ & & \frac{z y_0^{p_0/z+p_1/z}y_0^{d_0} y_1^{d_1}\prod_{m=0}^{b d_0-1}(\lambda_3-b p_0-m z)}{\prod_{m=1}^{d_0-d_1}(p_0-p_1+m z)\prod_{m=1}^{(b-1)d_0+d_1}((b-1)p_0+p_1+m z)}\Bigg]\,,
\eea  
where $p_0 \coloneqq H/(ab)$, $p_1 \coloneqq H/(ab)-E$. It is immediate to check that there are no non-trivial contributions to the mirror map,
\beq
[z^0] I^{\exIII}(y,z) = (\log y_0) p_0+ (\log y_1) p_1 \,,
\eeq 
and so we obtain
\beq 
J_{\rm small}^{\exIII}(t,z)=I^{\exIII}(\re^t,z)\,.
\eeq 
For the $\cO(1/z)$ term in the expansion of the $J$-functions we have quantum corrections only when $d_0=0$ or $d_1= (1-a) d_0$:
\bea 
[z^{-1}] I^{\exIII}(y,z) &=&
\frac{ \left((1-a b) \log ^2y_1+\log ^2y_0+2 \log y_1 \log
   y_0\right)}{2 a b}\mathrm{pt} \nn \\
  &+& \sum_{d>0} \frac{(-1)^{(b+1) d} \left(p_0-\lambda _1\right) (b d-1)! \left(b p_0-\lambda _3\right) y_1^{d-a d} y_0^d }{d z (a d)! 
   ((b -a) d)!} \nn \\
   &+& \sum_{d>0} \frac{\left(p_1-p_0\right)  \left((a-1) p_0-\lambda _2+p_1\right)y_1^d}{d^2}\,.
   \label{eq:Ifunzm1}
\eea
Acting on \eqref{eq:Ifunzm1} with $ \theta_0 ((b-1) \theta_0+\theta_1)$, where $\theta_i \coloneqq y_i \partial_{y_i}$, annihilates the quantum corrections in the last two rows. This entails that the small quantum cohomology product of $p_0$ and $(b-1) p_0+p_1$ is equal to their classical cup product:
\beq 
p_0 \star_y ((b-1) p_0+p_1) =  p_0 \cup ((b-1) p_0+p_1) = \frac{H^2}{a^2 b}=\frac{\mathrm{pt}}{a}\,.
\eeq 
At the next order in $\cO(1/z)$, we have
\bea
& & [z^{-2}] I^{\exIII}(y,z) =
\nn \\ & &\sum_{(d_0,d_1)\neq (0,0)}\frac{y_0^{d_0} y_1^{d_1} (-1)^{d_0 (a+b)+d_1} \left(b d_0-1\right)! \left(p_0-\lambda _1\right)  \left((a-1)
   p_0-\lambda _2+p_1\right) \left(b p_0-\lambda _3\right)}{d_0 \left((a-1) d_0+d_1\right)  \left(d_0-d_1\right)! \left((b-1) d_0+d_1\right)!}\,,
\eea
so that
\bea
N^{\rm loc, \psi}_{d}\l(\exIII\r) &=& \eta_{{\exIII}}\l({\rm pt}, [z^{-2}y_0^{d_0} y_1^{d_1}] I^{\exIII}\r) \nn \\ &=& \frac{(-1)^{d_0 (a+b)+d_1+1}  \left(a d_0-1\right)!}{d_0 \left((b-1) d_0+d_1\right) \left(d_0-d_1\right)! \left((a-1)
   d_0+d_1\right)!}\,,
\eea 
and thus
\bea
N^{\rm loc}_{d}\l(\exIII\r) &=& \eta_{{\exIII}}\l({\rm pt}, [z^{-2} y_0^{d_0} y_1^{d_1}] b \theta_0 ((a-1) \theta_0+\theta_1) I^{\exIII}\big)\r) \nn \\ &=& a d_0((b-1)d_0+d_1) N^{\rm loc, \psi}_{d}\l(\exIII\r)\,.
\label{eq:locIII}
\eea 

\subsection{Log GW invariants}

\begin{figure}[t]
\begin{tikzpicture}[smooth, scale=1.2]
\draw[step=1cm,gray,very thin] (-2.5,-2.5) grid (2.5,3.5);
\draw[<->] (-2.5,3) to (-1.5,3);
\draw[<->] (-2,2.5) to (-2,3.5);
\node at (-1.67,3.15) {$\scriptstyle{x}$};
\node at (-2.15,3.3) {$\scriptstyle{y}$};
%
%
\draw[<->,thick] (-2.5,-0.5) to (0.5,-0.5) to (0.5,-2.5);
\draw[->,thick] (0.5,-0.5) to (1,1) to (2.2,3.4);
\draw (-2.5,1) to (2.5,1);
\draw (-0.5,-2.5) to (-0.5,1);
\draw (-0.5,1) to (0.7,3.4);
\node at (-0.8,-2.3) {$\overline{D_3}$};
\node at (0.3,3.2) {$\overline{D_1}$};
\node at (-2.3,1.25) {$\overline{D_2}$};
%
\node at (-1.8,1) {$\times$};
%
\node at (-1.15,1.15) {$\sstyle{1+tx^{-1}}$};
%
\node at (-1.6,-0.3) {$\sstyle{x^{-(b-1)d_0-d_1}}$};
\node at (0.9,-2.2) {$\sstyle{y^{-ad_0}}$};
\node at (2.2,2.2) {$\sstyle{x^{bd_0}y^{ad_0}}$};
\node at (0.8,1.2) {$\circleed{1}$};
\node at (0.5,-0.5) {$\bullet$};
\node at (0.4,-0.35) {$\sstyle{p}$};
\end{tikzpicture}
\caption{$\scatt \exIII$}
\label{fig:Yab3comp}
\end{figure}

\begin{prop}
Denote by $N^{{\rm log},\psi}_d\left(\exIII\right)$ the genus 0 log GW invariant of $\exIII$ with a point class with psi insertion \cite[(4.2)]{BBvG2}. Then
\bea
N^{{\rm log},\psi}_d\left(\exIII\right) &= &\dbinom{ad_0}{d_0-d_1}\,,\\
\mathsf{N}^{\rm log}_d\left(\exIII\right)(\hbar) &= & \dfrac{[ad_0\left( (b-1)d_0+d_1 \right)]_q}{[1]_q}  \left[\begin{array}{c} ad_0 \\ d_0-d_1\end{array}\right]_q\,.
\label{eq:NdlogIII}
\eea
\label{prop:logIII}
\end{prop}

The genus zero ($q \to 1$) limit of \cref{prop:logIII}, combined with \eqref{eq:locIII}, concludes the proof of \cref{thm:log-local}.

\begin{proof}
Since $\exIII$ is obtained from $\bbP(1,a,b)$ by blowing up a smooth point on $D_{(-1,0)}$,  the scattering diagram of $\exIII$ is given by the fan of $\bbP(1,a,b)$ with a focus-focus singularity in the direction $(-1,0)$. By \cite[Section 4.2]{BBvG2}, we now need to extract the identity component of the multiplication of three broken lines corresponding to the directions $D_i$ with weights $d\cdot D_i$, $i=1,2,3$. The broken line calculation is given in \cref{fig:Yab3comp}. Choosing the location of $p$ as in \cref{fig:Yab3comp}, there is only one possible wall-crossing for the broken line coming from the $D_1$-direction.
After crossing the wall at $\circleed{1}$, the broken line coming from the $D_1$-direction carries the monomial 
\[
\qbinom{ad_0}{d_0-d_1}_qt^{d_0-d_1}x^{(b-1)d_0+d_1}y^{ad_0}\,.
\]
The result then follows from \cite[Proposition 4.1]{BBvG2} for the invariant with psi class and by \cite[Proposition 4.2]{BBvG2} for the invariant with two point classes.
\end{proof}

\subsection{Open GW invariants}

For $\exIII$ the relevant open geometry is  a toric Lagrangian triple given by affine space $\bbC^3$ with two toric Lagrangians $L_1$, $L_2$ at framing $f_1=(a/b-1)$ and $f_2=0$; see \cref{fig:annp11n}.\footnote{A strict application of \cite[Construction 6.4]{BBvG2} would in fact return $\bbC^3$ with two toric Lagrangians at framing $(-1,-a/b)$; the equivalence of these two open Gromov--Witten setups is a consequence of the well-known symmetries of the topological vertex, see \cite[Section~3.4]{Aganagic:2003db}, as can also be easily verified in the foregoing formulas.} For this setup, denoting $j_i$ the winding number of open stable maps around $S^1 \hookrightarrow L_i$, $i=1,2$, we have
\bea 
\mathsf{O}_{j_1,j_2}\l((\exIII)^{\rm op}\r) = 
\sum_{i_1=0}^{j_1-1}\sum_{i_2=0}^{ j_2-1} \frac{(-1)^{i_1+i_2}}{j_1 j_2} \cW^{(c)}_{(j_1-i_1,1^{i_1}),( j_2-i_2,1^{i_2})}\l((\exIII)^{\rm op}\r) \,,
\label{eq:annYab1}
\eea 
\begin{figure}[t]
\begingroup%
  \makeatletter%
  \providecommand\color[2][]{%
    \errmessage{(Inkscape) Color is used for the text in Inkscape, but the package 'color.sty' is not loaded}%
    \renewcommand\color[2][]{}%
  }%
  \providecommand\transparent[1]{%
    \errmessage{(Inkscape) Transparency is used (non-zero) for the text in Inkscape, but the package 'transparent.sty' is not loaded}%
    \renewcommand\transparent[1]{}%
  }%
  \providecommand\rotatebox[2]{#2}%
  \newcommand*\fsize{\dimexpr\f@size pt\relax}%
  \newcommand*\lineheight[1]{\fontsize{\fsize}{#1\fsize}\selectfont}%
  \ifx\svgwidth\undefined%
    \setlength{\unitlength}{144.3133569bp}%
    \ifx\svgscale\undefined%
      \relax%
    \else%
      \setlength{\unitlength}{\unitlength * \real{\svgscale}}%
    \fi%
  \else%
    \setlength{\unitlength}{\svgwidth}%
  \fi%
  \global\let\svgwidth\undefined%
  \global\let\svgscale\undefined%
  \makeatother%
  \begin{picture}(1,0.95865837)%
    \lineheight{1}%
    \setlength\tabcolsep{0pt}%
    \put(0,0){\includegraphics[width=\unitlength,page=1]{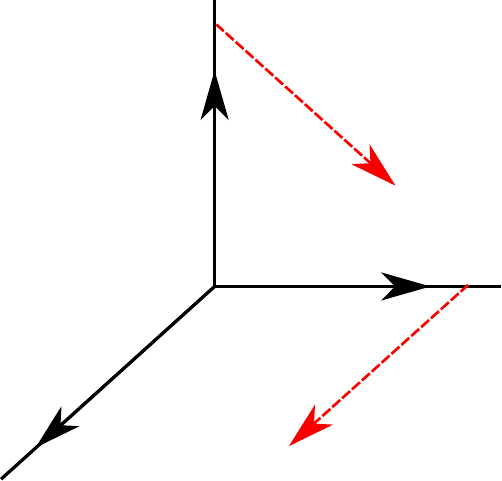}}%
    \put(0.56286927,0.82845202){\makebox(0,0)[lt]{\lineheight{1.25}\smash{\begin{tabular}[t]{l}$\mathsf{f}_1$\end{tabular}}}}%
    \put(0.62637676,0.22995036){\makebox(0,0)[lt]{\lineheight{1.25}\smash{\begin{tabular}[t]{l}$\mathsf{f}_2$\end{tabular}}}}%
  \end{picture}%
\endgroup%

\caption{The toric graph of $\big(\exIII\big)^{\rm op}$ for $a=2$, $b=1$.}
\label{fig:annp11n}
\end{figure}
where $\cW^{(c)}_{\a\b}$ is the connected 2-leg topological vertex at framing $(a/b-1,0)$,
\bea 
\cW^{(c)}_{\a\b}\l((\exIII)^{\rm op}\r) & \coloneqq &
\cW_{\a\b}\l((\exIII)^{\rm op}\r)-
\cW_{\a\emptyset}\l((\exIII)^{\rm op}\r)
\cW_{\emptyset\b}\l((\exIII)^{\rm op}\r) \nn \\
&=&
q^{a/b\kappa(\a)/2}(-1)^{(a/b-1)|\a|} \sum_{\delta \neq \emptyset} s_{\frac{\a^t}{\delta}}(q^\rho) s_{\frac{\b}{\delta}}(q^\rho)\,.
\eea 
We can express \eqref{eq:annYab1} as a $q$-series using the $q$-factorial expression for hook skew Schur functions \cite[App.~C]{BBvG2},
\beq 
s_{\frac{(d,1^{i})}{\delta}}(q^\rho) = 
\l\{
\begin{array}{cc}
\frac{(-1)^{-d-i+k+r} q^{\frac{1}{2} \left(d^2-2 d k+i+k^2-r\right)}}{(q;q)_{d-k} (q;q)_{i-r}}, & \delta = (k,1^{r})\,, \\ 
0 & \rm else\,,
\eary 
\r.
\eeq 
as
\bea 
 \mathsf{O}_{j_1,j_2}\l((\exIII)^{\rm op}\r) &=& \frac{(-1)^{a/b j_1+j_2+1} q^{\frac{1}{2} \left(\frac{a j_1^2}{b}+j_1 \left(\frac{a}{b}+1\right)+j_2^2\right)}}{j_1 j_2}\sum_{k,l_1,l_2=0}^\infty \sum_{r=0}^{k-1} q^{ \frac{arj_1}{b}} \mathsf{a}_{l_1,j_1,k} \mathsf{b}_{l_2,j_2,k} \mathsf{c}_{k} \,, \nn \\
\eea 
where we have shifted the indices of summation as  $i_1=l_1+k-r-1$, $i_2=l_2+r$, and moreover
\bea  
\mathsf{a}_{l_1,j_1,k} &\coloneqq & \frac{(-1)^{l_1} q^{\frac{1}{2} l_1 \left(l_1-1-2a/b  j_1\right)}}{(q;q)_{l_1} (q;q)_{j_1-k-l_1}} \,, \nn \\
\mathsf{b}_{l_2,j_2,k} &\coloneqq & \frac{(-1)^{l_2} q^{\frac{1}{2} l_2 \left(l_2+1+2k-2 j_2\right)}}{(q;q)_{l_2} (q;q)_{j_2-k-l_2}} \nn \,, \\
\mathsf{c}_{k} &\coloneqq & (-1)^k q^{\frac{1}{2} k \left(k-1-2a/b j_1-2 j_2\right)}\,.
\eea 
In the formulas above, $(q;q)_n$ is the usual $q$-Pochhammer symbol, $(q;q)_n \coloneqq (-1)^n q^{\frac{1}{4} n (n+1)} [n]_q!$. Performing the $l_2$ summation using the Cauchy binomial theorem in the form
\beq 
(q y; q)_n  = \sum_{m=0}^\infty \frac{(-y)^m q^{\frac{m(m+1)}{2}} (q;q)_n}{(q;q)_m (q;q)_{n-m}}
\label{eq:cauchy}
\eeq 
gives 
\beq 
\sum_{l_2=0}^\infty \mathsf{b}_{l_2,j_2,k} = \frac{(q^{k-j_2+1}; q)_{j_2-k}}{(q; q)_{j_2-k}} = \frac{1}{(q; q)_{j_2-k} (q; q)_{k-j_2}}=\delta_{j_2k}\,.
\eeq 
The sum over $k$ consists then of a single summand at $k=j_2$, and the sum over $l_1$, using \eqref{eq:cauchy}, is
\beq
\sum_{l_1=0}^{j_1-j_2} \mathsf{a}_{l_1,j_1,j_2} 
=
(-1)^{j_1-j_2} q^{\frac{1}{2}(j_2-j_1) (j_1 a/b+1)}\qbinom{aj_1/b}{j_1-j_2}_q
\eeq 
Therefore,
\bea 
\mathsf{O}_{j_1,j_2}\l((\exIII)^{\rm op}\r) &=& 
\frac{(-1)^{j_1 (1+a/b)+1+j_2}}{j_1 j_2} q^{-\frac{a j_1 \left(j_2-1\right)}{2 b}} \qbinom{ aj_1/b}{j_1-j_2}_q
\sum_{r=0}^{j_2-1} q^{rj_1 a/b} \nn \\
&=& 
(-1)^{j_1 (1+a/b)+j_2+1} \qbinom{ aj_1/b}{j_1-j_2}_q \frac{[j_1 j_2 a/b]_q}{[j_1 a/b]_q  j_1 j_2}
 \,.
\label{eq:openIII}
\eea 
From \eqref{eq:iota},
the winding number variables in $(\exIII)^{\rm op}$ and curve degrees in $Y^{[3]}_{(a,b)}$ are related as
\bea 
j_1 & \to & b d_0\,, \nn \nn \\ 
j_2 & \to & (b-1) d_0+d_1\,,
\eea  
which combined with \eqref{eq:openIII} and \cref{prop:logIII} concludes the proof of \cref{thm:logopen}.

\section{BPS invariants}

The closed-form higher genus GW expressions of the previous Sections put us now in a position to prove \cref{thm:openbps}. Write 
\beq 
 \Omega_{d}(Y(D))(q)  =
 \frac{[1]_{q}^2}{\prod_{i=1}^l[ d \cdot D_i]_{q}}
\sum_{k|d}(-1)^{\sum_{i=1}^{l}  d/k \cdot D_i+1}\mu(k)
\l(\frac{[k]_q}{k}\r)^{3-l}\mathsf{N}^{\rm log}_{d/k}(-\ri k \log q)
\,.
\label{eq:Omegad3}
\eeq 
From \eqref{eq:NdlogI}, \eqref{eq:NdlogII} and \eqref{eq:NdlogIII}, we have that obviously $\mathsf{N}^{\rm log}_{d/k}(-\ri k \log q) \in \bbZ[q^{\pm 1/2}]$ since it is a product of $q$-binomial coefficients. Then \eqref{eq:Omegad3} implies {\it a priori} that $\Omega_d(Y(D))(q)\in \bbQ(q^{1/2})$ with poles at $q=0,\infty$ and at most double poles at $q=\exp(2\pi \ri l/\hat d)$, with $\hat d \coloneqq \mathrm{lcm}\{ d \cdot D_i \}_{i=1}^l$. We have the following
\begin{prop}
Let $Y(D)=\exI$, $\exII$, $\exIII$. Then 
\beq 
\frac{\prod_{i=1}^l[ d \cdot D_i]_{q}}{[1]_{q}^2}\Omega_{d}(Y(D))(q) = \cO\l(q-\re^{\frac{2\pi \ri l}{\hat d}}\r)^2, \quad l=0, \dots, \hat d-1.
\eeq 
\end{prop}
\begin{proof}
The vanishing at linear order can be shown with the exact same arguments of the proof of \cite[Thm~8.1]{BBvG2} by replacing therein $\Theta_d(q) \to \mathsf{N}^{\rm loc}_{d}(-\ri \log q)$ (for $\exI$ and $\exII$) and $\Xi_d(q) \to \mathsf{N}^{\rm loc}_{d}(-\ri \log q)$ (for $\exIII$): the summands in the divisor sums in \eqref{eq:Omegad3} can be grouped in pairs with leading order terms at $q=\re^{\frac{2\pi \ri l}{\hat d}}$ having opposite signs, ensuring the l.h.s. is zero at that order. The quadratic vanishing is a consequence of \cite[Lemma~8.3]{BBvG2}. 
\end{proof}

The Proposition then implies that $\Omega_d(Y(D)) \in \bbQ[q^{\pm1/2}]$. Since $1/[d \cdot D_i]_q \in q^{-d \cdot D_i/2}\bbZ[[q]]$, from \eqref{eq:Omegad3} we have $\Omega_d(Y(D)) \in \bbZ[q^{-1/2}][[q^{1/2}]]$, and thus $\Omega_d(Y(D)) \in \bbZ[q^{\pm1/2}]$ from the previous Proposition. The claim of \cref{thm:openbps} then follows.

\subsection{Quiver DT invariants}

\begin{figure}[t]
    \centering
    \includegraphics[scale=1.3]{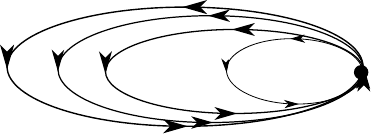}
    \caption{The quiver for $\exI$ for $a=1$, $b=3$.}
    \label{fig:vertquiv}
\end{figure}

For $l=2$ and whenever $a=1$ or $b=1$, $Y^{\rm op}(D)$ is an integrally framed toric Lagrangian triple of `strip' type \cite{Panfil:2018sis},  that is, it consists of a single integrally framed outer Aganagic--Vafa Lagrangian on a smooth toric Calabi--Yau threefold whose fan is a cone over an integral polytope of unit width. Then \cref{thm:kpdt} follows from a proof identical in all its parts to that of \cite[Thm~7.3]{BBvG2}, with framings now equal to $f=b$ (for $\exI$ with $a=1$) and $f=b-1$ (for $\exII$ with $a=1$).

\begin{example}
For $Y(D)=\bbP(1,1,b)^{[2]}$, $Y^{\rm op}(D)$ is the 1-legged vertex at framing $b$, for which the corresponding quiver is the $b+1$-loop quiver \cite{Panfil:2018sis} (see \cref{fig:vertquiv}). The dimension vector is here identically identified with the curve degree $d$, $\kappa=\mathrm{id}$, and the integral shift in \cref{thm:kpdt} vanishes, $\a_i=0$. The Klemm--Pandharipande invariants $\mathrm{KP}_d(E_{\bbP(1,1,b)^{[2]}})$ are then up to a sign the polynomials in $\frac{1}{d!}\bbZ[b]$ computed by Reineke in \cite[Thm~3.2]{MR2889742}. Explicitly, we have
\bea
\mathrm{KP}_d\big(E_{\bbP(1,1,b)^{[2]}}\big) &=&
\bigg\{(-1)^b,\frac{1}{4}  \left( (2 b+1)-(-1)^b\right),\frac{1}{2} (-1)^b b (b+1),\nn \\ 
& & \frac{1}{3} b (b+1) (2 b+1),\frac{5}{24} (-1)^b b (b+1) (5 b
   (b+1)+2), \dots\bigg\}
   \nn \\
   &=& (-1)^{b+d+1}\mathrm{DT}_d\big(\mathsf{Q}(\bbP(1,1,b)^{[2]})\big)
\eea 
\end{example}

\begin{example}
For $Y(D)=Y_{(1,b)}^{[2]}$, the corresponding quiver is given in \cref{fig:conifquiv}. \\

\begin{figure}[t]
    \centering
    \includegraphics[scale=1.5]{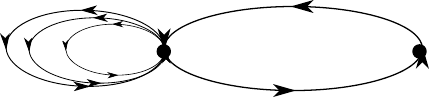}
    \caption{The quiver for $\exII$ for $a=1$, $b=3$.}
    \label{fig:conifquiv}
\end{figure}
The map between vertices of the quiver and effective generators of $\hhh_2(Y_{(1,b)}, \bbZ)$ is 
\bea 
v_1 & \stackrel{\kappa}{\longrightarrow} & f \nn \\
v_2 & \stackrel{\kappa}{\longrightarrow} & E
\eea 
and the integral shifts in \cref{thm:kpdt} are $\a_1=0$, $\a_2=1$.  Klemm--Pandharipande invariants for small degrees $d_0, d_1>0$ are given in \cref{tab:DTconif}: note that despite being rational polynomials in $b$, they take integer values for $b\in \bbZ^+$. The quiver DT invariants of $\mathsf{Q}(Y_{(1,b)}^{[2]})$ are obtained as their absolute values.
\begin{table}[h]
    \centering
    \begin{tabular}{|c|c|c|c|c|}
    \hline 
\diagbox{$d_0$}{$d_1$} & 1 & 2 & 3 & 4 \\ \hline 
1 & $(-1)^b$ & 0 & 0 & 0  \\ \hline
2 &
 $-b$ & $\frac{(-1)^b\left((-1)^b (2 b+1)-1\right)}{4}  $ & 0 & 0  \\ \hline
 3 &
 $\frac{(-1)^b b (3 b-1)}{2} $ & $-\frac{(-1)^b b (3 b+1)}{2} $ & $\frac{ (-1)^b b (b+1)}{2}$ & 0  \\ \hline
 4 &
 $-\frac{b\left(8 b^2-6 b+1\right)}{3}$ & $4 b^3$ & $-\frac{b (2 b+1) (4 b+1)}{3} $ & $\frac{b (b+1) (2 b+1)}{3} $  \\ \hline
 5 &
 $\frac{(-1)^b b (5 b-3) (5 b-2) (5 b-1)}{24} $ & $\frac{(-1)^{b+1} b (5 b-2) (5 b-1) (5 b+1)}{12} $ & $\frac{(-1)^b b (5 b-1) (5 b+1) (5 b+2)}{12} $ &
   $\frac{(-1)^{b+1} b (5 b+1) (5 b+2) (5 b+3)}{24}$   \\ \hline
    \end{tabular}
    \caption{BPS/KP invariants of $Y_{(1,b)}^{[2]}$. }
    \label{tab:DTconif}
\end{table}
\end{example}

\subsection*{Acknowledgements} Boris Dubrovin has been a figure of immense importance for the fields of Geometry and Mathematical Physics in general, and an attentive and kind mentor for some of us in particular. Even if centred in a topic far from the theory of integrable systems, this paper bears a very direct intellectual debt to Boris' work (most crucially, in the 2-point reconstruction lemma proving \cref{thm:log-local} for $\exIII$). It is a privilege to be able to dedicate this paper as a modest testament to his legacy.

\bibliography{miabiblio}

\end{document}